\documentclass{amsart}
\usepackage[hypertex]{hyperref}
\usepackage{amsmath,amssymb}
\usepackage{bbold}
\usepackage{psfrag}
\usepackage{graphicx}
\usepackage{leftidx}
\usepackage[all]{xy}
\usepackage[dvipsnames]{xcolor}

\newcommand{\kt}{$\Bbbk$\nobreakdash-\hspace{0pt}}
\newcommand{\Ob}{\mathrm{Ob}}
\newcommand{\Hn}{H_\param^n}
\newcommand{\co}{\colon}
\newcommand{\sss}{\mathcal{S}}
\newcommand{\up}{\mathcal{U}}
\newcommand{\low}{\mathcal{L}}
\newcommand{\I}{\mathcal{I}}
\newcommand{\un}{\mathbb{1}}
\newcommand{\param}{\omega}
\newcommand{\End}{\mathrm{End}}
\newcommand{\Hom}{\mathrm{Hom}}
\newcommand{\tr}{\mathrm{tr}}
\newcommand{\id}{\mathrm{id}}
\newcommand{\Z}{\mathbb{Z}}
\newcommand{\N}{\mathbb{N}}
\newcommand{\R}{\mathbb{R}}
\newcommand{\T}{\mathbb{T}}
\newcommand{\kk}{\Bbbk}
\newcommand{\ns}{\theta}
\newcommand{\mti}{\,\mbox{-}\,}
\newcommand{\inv}{\tau}
\newcommand{\reve}{\mathrm{rev}}
\newcommand{\opp}{\mathrm{op}}
\newcommand{\mand}{\quad \text{and} \quad}
\newcommand{\lev}{\mathrm{ev}}
\newcommand{\lcoev}{\mathrm{coev}}
\newcommand{\rev}{\widetilde{\mathrm{ev}}}
\newcommand{\rcoev}{\widetilde{\mathrm{coev}}}
\newcommand{\labeli}{\renewcommand{\labelenumi}{{\rm (\roman{enumi})}}}
\newcommand{\labela}{\renewcommand{\labelenumi}{{\rm (\alph{enumi})}}}
\newcommand{\rsdraw}[3]{\raisebox{-#1\height}{\scalebox{#2}{\includegraphics{#3.eps}}}}
\newcommand{\norm}[1]{\left\langle#1\right\rangle}

\allowdisplaybreaks[3]

\newtheorem{theorem}{Theorem}
\newtheorem{lemma}{Lemma}
\newtheorem{claim}{Claim}

\title{Generalized Kuperberg invariants of 3-manifolds}
\author{Rinat Kashaev}
\address{Section de math\'ematiques, Universit\'e de Gen\`eve,
2-4 rue du Li\`evre, 1211 Gen\`eve 4, Suisse\\}
\email{rinat.kashaev@unige.ch}
\author{Alexis Virelizier}
\address{D\'epartement de math\'ematiques, Universit\'e de Lille,
Cit\'e Scientifique, 59655 Villeneuve d'Ascq, France\\}
\email{alexis.virelizier@univ-lille.fr}

\date{\today}

\thanks{Supported in part by Swiss National Science Foundation.}
\begin{document}
\begin{abstract}
In the 90s, based on presentations of 3-manifolds by Heegaard diagrams, Kuperberg associated a scalar invariant of 3-manifolds  to each  finite dimensional involutory Hopf algebra over a field.  We generalize this construction to the case of involutory Hopf algebras in arbitrary symmetric monoidal categories  admitting certain pairs of morphisms called \emph{good pairs}.  We construct examples of such good pairs for involutory Hopf algebras whose distinguished grouplike elements are central. The generalized construction is illustrated by an example of an involutory super Hopf algebra.
\end{abstract}
\maketitle
\section{Introduction}
In the 90s, G.~Kuperberg in \cite{Ku1} defined a scalar invariant of closed oriented  3-manifolds by using arbitrary finite-dimensional involutory Hopf algebra over a field. The combinatorics of Kuperberg's construction comes from the presentation of 3-manifolds by Heegaard diagrams and the invariant is calculated by using the tensor products of the structural morphisms of the Hopf algebra contracted with appropriate tensor products of (two-sided) integrals and cointegrals of the Hopf algebra.

In this paper, we generalize Kuperberg's construction to the context of involutory Hopf algebras in arbitrary symmetric monoidal categories where the tensor products of structural morphisms are contracted with morphisms which, in general, are different from integrals and cointegrals (at least when the latter are not two-sided).  More precisely, we define the notion of a \emph{good pair} of an involutory Hopf algebra in a symmetric monoidal category and prove that such a pair gives rise to a topological invariant of closed oriented 3-manifolds (see Theorem~\ref{main-thm}). Also, we give a way to derive examples of good pairs from integrals and cointegrals  when the distinguished grouplike elements are central  (see Theorem~\ref{thm-gp}). In the particular case where the integrals and cointegrals are two-sided (i.e., the distinguished grouplike elements and the distinguished algebra morphisms are trivial), the associated good pair is given by the integral and cointegral themselves, reproducing thereby Kuperberg's original construction.

The paper is organized as follows. In Section~\ref{sec1}, we describe the notation and conventions used and specify our setup which  is given by a symmetric monoidal category, together with Hopf  algebras  and invertible objects in it. In Section~\ref{sec2}, based on Kuperberg's work \cite{Ku1}, first we recall the combinatorics of Heegaard diagrams associated to 3-manifolds and then explain how, by using the structural maps of an involutory  Hopf  algebra  in a symmetric monoidal category, one associates a tensor to each  (ordered, oriented, and based)  Heegaard diagram. In Section~\ref{sec3}, we introduce the notion of a good pair and prove
that such a pair gives rise to a 3-manifold invariant (Theorem~\ref{main-thm}). Also, we illustrate the general construction with an example of an involutory super Hopf algebra (i.e., an involutory Hopf algebra in the symmetric monoidal category of super vector spaces). In Section~\ref{sec4}, we describe  how one  can construct good pairs starting from integrals and cointegrals  (Theorem~\ref{thm-gp}).

\section{Algebraic preliminaries}\label{sec1}

In this section, $\sss$ is a symmetric monoidal category, with monoidal product $\otimes$, unit object $\un$, and symmetry
$\tau=\{\tau_{X,Y} \co X \otimes Y \to Y \otimes X\}_{X,Y \in \Ob(\sss)}$.
Recall that $\End_\sss(\un)$ is a commutative monoid (with unit $1=\id_\un$) which acts on each Hom set by the monoidal product. Since $\sss$ is symmetric, this action satisfies that
$$
\alpha (f \otimes g)=(\alpha f) \otimes g=f \otimes (\alpha g)
$$
for all morphisms $f,g$ in $\sss$ and all $\alpha \in \End_\sss(\un)$.
The elements of $\End_\sss(\un)$ are called \emph{scalars}.  For objects $X,Y$ of $\sss$, the notation $X \simeq Y$ means that $X$ and $Y$ are isomorphic, i.e., that  there is an invertible morphism $X\to Y$.

\subsection{Conventions}\label{sect-conventions}
We  will suppress in our formulas  the associativity  and unitality  constraints   of the monoidal category $\sss$.   This does not lead  to   ambiguity because by   Mac Lane's   coherence theorem (see \cite{ML}),   all legitimate ways of inserting these constraints   give the same results.
For  any objects $X_1,...,X_n$ with $n\geq 2$, we set
$$
X_1 \otimes X_2 \otimes \cdots \otimes X_n=(... ((X_1\otimes X_2) \otimes X_3) \otimes
\cdots \otimes X_{n-1}) \otimes X_n
$$
and similarly for morphisms.

\subsection{Graphical conventions}

Morphisms in $\sss$ may be represented by diagrams to be read from bottom to top.  We discuss here the basics of this Penrose graphical calculus. The identity $\id_X$ of an  object $X$ of $\sss$, a morphism $f\co X \to Y$ in~$\sss$, and the composition of two morphisms $f\co X \to Y$ and $g\co Y \to Z$ may be graphically represented as follows:
$$
\psfrag{X}[Bl][Bl]{\scalebox{.8}{$X$}} \psfrag{Y}[Bc][Bc]{\scalebox{.8}{$Y$}} \psfrag{h}[Bc][Bc]{\scalebox{.9}{$f$}} \psfrag{g}[Bc][Bc]{\scalebox{.9}{$g$}}
\psfrag{Z}[Bc][Bc]{\scalebox{.8}{$Z$}} \id_X=\, \rsdraw{.45}{.9}{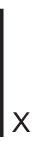}\,,\quad f=\, \rsdraw{.45}{.9}{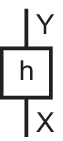}\;\,,\quad \text{and} \quad gf=\, \rsdraw{.45}{.9}{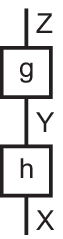}\,\; .
$$
The monoidal product of two morphisms $f\co X \to Y$
and $g \co U \to V$ in~$\sss$ is represented by juxtaposition:
$$
\psfrag{X}[Bl][Bl]{\scalebox{.8}{$X$}}
\psfrag{h}[Bc][Bc]{\scalebox{.9}{$f$}}
\psfrag{Y}[Bl][Bl]{\scalebox{.8}{$Y$}}
f\otimes g= \,\rsdraw{.45}{.9}{morphism-bt}\,
\psfrag{X}[Bl][Bl]{\scalebox{.8}{$U$}}
\psfrag{g}[Bc][Bc]{\scalebox{.9}{$g$}}
\psfrag{Y}[Bl][Bl]{\scalebox{.8}{$V$}}
\rsdraw{.45}{.9}{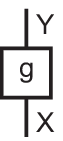}\,\;   .
$$
We can also use boxes with several strands attached to their horizontal sides. For example, a morphism $f \co X \otimes Y \to A \otimes B \otimes C$ with  $X,Y,A,B,C\in \Ob (\sss)$   may be represented in various ways,   such   as
$$
\psfrag{X}[Bl][Bl]{\scalebox{.8}{$X$}}
\psfrag{h}[Bc][Bc]{\scalebox{.9}{$f$}}
\psfrag{Y}[Bl][Bl]{\scalebox{.8}{$Y$}}
\psfrag{A}[Bl][Bl]{\scalebox{.8}{$A$}}
\psfrag{B}[Bl][Bl]{\scalebox{.8}{$B$}}
\psfrag{C}[Bl][Bl]{\scalebox{.8}{$C$}}
\rsdraw{.45}{.9}{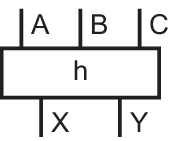}
\qquad \text{or} \qquad
\psfrag{X}[Bl][Bl]{\scalebox{.8}{$X \otimes Y$}}
\psfrag{h}[Bc][Bc]{\scalebox{.9}{$f$}}
\psfrag{A}[Bl][Bl]{\scalebox{.8}{$A$}}
\psfrag{B}[Bl][Bl]{\scalebox{.8}{$B \otimes C$}}
\rsdraw{.45}{.9}{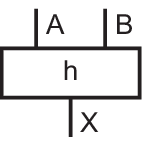}
\qquad \text{or} \qquad
\psfrag{X}[Bl][Bl]{\scalebox{.8}{$X$}}
\psfrag{h}[Bc][Bc]{\scalebox{.9}{$f$}}
\psfrag{Y}[Bl][Bl]{\scalebox{.8}{$Y$}}
\psfrag{A}[Br][Br]{\scalebox{.8}{$A \otimes B$}}
\psfrag{B}[Bl][Bl]{\scalebox{.8}{$C$}}
\rsdraw{.45}{.9}{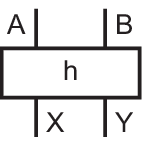} \,\;.
$$
In accordance with   conventions of  Section~\ref{sect-conventions}, we  ignore here the associativity   constraint   between the  objects $A \otimes B \otimes C =(A \otimes B)\otimes C$ and $A \otimes (B\otimes C)$.
A~box whose lower/upper side has no  attached strands represents a morphism with source/target $\un$. For example,
morphisms $\alpha\co \un \to \un$,   $\beta\co \un \to X$,   $\gamma\co X \to \un$  with $X\in \Ob (\sss)$ may be represented by the diagrams
$$
\psfrag{h}[Bc][Bc]{\scalebox{.9}{$\alpha$}}
\rsdraw{.45}{.9}{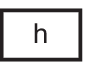}\;, \qquad
\psfrag{X}[Bl][Bl]{\scalebox{.8}{$X$}}
\psfrag{h}[Bc][Bc]{\scalebox{.9}{$\beta$}}
\psfrag{Y}[Bl][Bl]{\scalebox{.8}{$Y$}}
\rsdraw{.45}{.9}{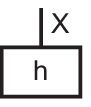}\;, \qquad
\psfrag{X}[Bl][Bl]{\scalebox{.8}{$X$}}
\psfrag{h}[Bc][Bc]{\scalebox{.9}{$\gamma$}}
\psfrag{Y}[Bl][Bl]{\scalebox{.8}{$Y$}}
\rsdraw{.45}{.9}{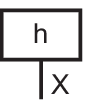}\;.
$$

Every diagram which is $\sss$-colored as above determines a morphism in $\sss$ obtained as follows. First slice the diagram into horizontal strips so that each strip is made of juxtaposition of vertical segments or boxes. Then, for each strip, take the monoidal product of the morphisms associated to the vertical segments or boxes. Finally,   compose the resulting morphisms  proceeding from the bottom to the top. For example, given $A,B,X,Y,Z \in \Ob(\sss)$ and morphisms $f\co Y \to Z$, $g\co B \otimes Z \to \un$, $h \co X \to A \otimes B$,  the diagram
$$
\psfrag{h}[Bc][Bc]{\scalebox{.9}{$g$}}
\psfrag{k}[Bc][Bc]{\scalebox{.9}{$f$}}
\psfrag{b}[Bc][Bc]{\scalebox{.9}{$h$}}
\psfrag{X}[Bl][Bl]{\scalebox{.8}{$X$}}
\psfrag{Y}[Bl][Bl]{\scalebox{.8}{$Z$}}
\psfrag{A}[Bl][Bl]{\scalebox{.8}{$A$}}
\psfrag{B}[Bl][Bl]{\scalebox{.8}{$B$}}
\psfrag{Z}[Bl][Bl]{\scalebox{.8}{$Y$}}
\rsdraw{.45}{.9}{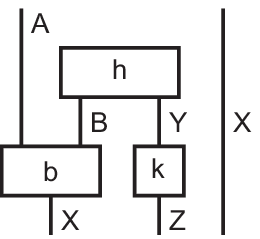}
$$
represents the morphism
$$
(\id_A \otimes g \otimes \id_X)(h \otimes f \otimes \id_X)=\bigl ( (\id_A \otimes g)(h \otimes f) \bigr ) \otimes \id_X
$$
from $X \otimes Y \otimes X$  to $A \otimes X$.

The functoriality of the monoidal product implies that the morphism associated to a $\sss$-colored diagram is  independent of the way of cutting it into horizontal strips. It also implies that we can push boxes lying on  the same horizontal level  up or down
without changing the morphism represented by the diagram. For example, for all morphisms $f\co X \to Y$ and $g\co U \to
V$ in~$\sss$, we have:
$$
\psfrag{X}[Bl][Bl]{\scalebox{.8}{$X$}}
\psfrag{f}[Bc][Bc]{\scalebox{.9}{$f$}}
\psfrag{Y}[Bl][Bl]{\scalebox{.8}{$Y$}}
\psfrag{U}[Bl][Bl]{\scalebox{.8}{$U$}}
\psfrag{g}[Bc][Bc]{\scalebox{.9}{$g$}}
\psfrag{V}[Bl][Bl]{\scalebox{.8}{$V$}}
\rsdraw{.45}{.9}{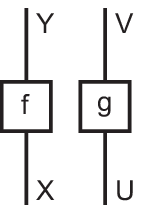} \, = \, \rsdraw{.45}{.9}{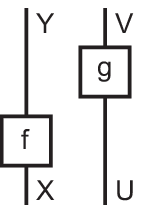} \, = \, \rsdraw{.45}{.9}{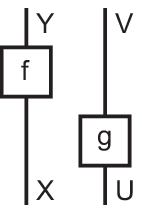}
$$
which graphically expresses the formulas
$$
f\otimes g =(\id_Y\otimes g) (f \otimes \id_U)= (f\otimes \id_V)( \id_X\otimes g).
$$
Here and in the sequel, the equality sign between  the diagrams     means the equality of the   corresponding morphisms.

The symmetry $\tau_{X,Y} \co  X \otimes Y\to Y \otimes X$ in $\sss$ is depicted as
$$
\psfrag{Z}[Bl][Bl]{\scalebox{.8}{$X$}}
\psfrag{X}[Br][Br]{\scalebox{.8}{$X$}}
\psfrag{A}[Bl][Bl]{\scalebox{.8}{$Y$}}
\psfrag{Y}[Br][Br]{\scalebox{.8}{$Y$}}
\tau_{X,Y}=\,\; \rsdraw{.45}{.9}{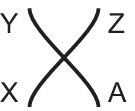} \,.
$$
The axioms of a symmetry are depicted as follows: for all $X,Y,Z \in \Ob(\sss)$,
\begin{center}
\begin{tabular}{cc}
  \begin{tabular}{@{}c@{}}
  \psfrag{Z}[Bl][Bl]{\scalebox{.8}{$X \otimes Y$}}
\psfrag{X}[Br][Br]{\scalebox{.8}{$X\otimes Y$}}
\psfrag{A}[Bl][Bl]{\scalebox{.8}{$Z$}}
\psfrag{Y}[Br][Br]{\scalebox{.8}{$Z$}}
\rsdraw{.45}{.9}{symmetry} \qquad = \;
\psfrag{Z}[Bl][Bl]{\scalebox{.8}{$X$}}
\psfrag{C}[Bl][Bl]{\scalebox{.8}{$Y$}}
\psfrag{X}[Br][Br]{\scalebox{.8}{$X$}}
\psfrag{B}[Br][Br]{\scalebox{.8}{$Y$}}
\psfrag{A}[Bl][Bl]{\scalebox{.8}{$Z$}}
\psfrag{Y}[Br][Br]{\scalebox{.8}{$Z$}}
\rsdraw{.45}{.9}{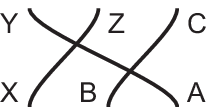} \;,\\[2em]
\psfrag{Z}[Bl][Bl]{\scalebox{.8}{$X$}}
\psfrag{X}[Br][Br]{\scalebox{.8}{$X$}}
\psfrag{A}[Bl][Bl]{\scalebox{.8}{$Y \otimes Z$}}
\psfrag{Y}[Br][Br]{\scalebox{.8}{$Y \otimes Z$}}
\rsdraw{.45}{.9}{symmetry} \qquad = \;
\psfrag{Z}[Bl][Bl]{\scalebox{.8}{$Y$}}
\psfrag{C}[Bl][Bl]{\scalebox{.8}{$Z$}}
\psfrag{X}[Br][Br]{\scalebox{.8}{$Y$}}
\psfrag{B}[Br][Br]{\scalebox{.8}{$Z$}}
\psfrag{A}[Bl][Bl]{\scalebox{.8}{$X$}}
\psfrag{Y}[Br][Br]{\scalebox{.8}{$X$}}
\rsdraw{.45}{.9}{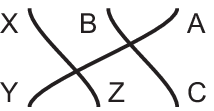} \;,
\end{tabular}
&
\psfrag{Y}[Bl][Bl]{\scalebox{.8}{$Y$}}
\psfrag{X}[Br][Br]{\scalebox{.8}{$X$}}
\psfrag{A}[Bl][Bl]{\scalebox{.8}{$X$}}
\psfrag{B}[Br][Br]{\scalebox{.8}{$Y$}}
\qquad \rsdraw{.45}{.9}{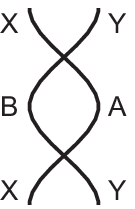} \; =
\psfrag{Y}[Bl][Bl]{\scalebox{.8}{$Y$}}
\psfrag{X}[Br][Br]{\scalebox{.8}{$X$}}
\rsdraw{.45}{.9}{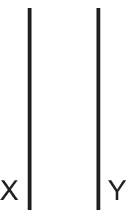} \;.
\end{tabular}
\end{center}
The naturality of $\tau$ ensures that we can push boxes across a strand without changing the morphism represented by the diagram: for any morphism $f$ in $\sss$,
$$
\psfrag{f}[Bc][Bc]{\scalebox{.9}{$f$}}
\rsdraw{.45}{.9}{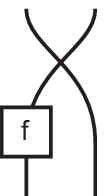} \; = \; \rsdraw{.45}{.9}{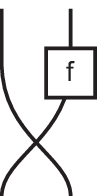} \quad \text{and} \quad
\rsdraw{.45}{.9}{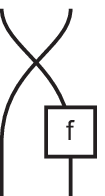} \; = \; \rsdraw{.45}{.9}{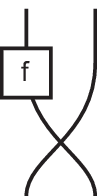} \;.
$$
For a detailed treatment of the Penrose graphical calculus, we refer to \cite{JS} or \cite{TV}.

\subsection{Representation of the symmetric group}\label{sect-repres-Sn}
Recall that a symmetric monoidal category provides representations of  the symmetric groups. More precisely, the axioms of a symmetry imply that to any $A \in \Ob(\sss)$ and $n \in \N$ is associated a unique group morphism
$$
\left \{ \begin{array}{ccc}\mathfrak{S}_n & \to & \End_\sss(A^{\otimes n}) \\ \sigma & \mapsto & P_\sigma \end{array} \right.
$$
such that for all~$1 \leq i \leq n-1$, the transposition $(i,i+1)$ is mapped to
$$
P_{(i,i+1)}=\id_{A^{\otimes (i-1)}} \otimes \tau_{A,A} \otimes \id_{A^{\otimes (n-i-1)}}.
$$

\subsection{Duality}\label{sect-duality}
Assume that $\sss$ is \emph{left rigid}, i.e. to each object $X \in \Ob(\sss)$ is associated an object $X^* \in \Ob(\sss)$, called the \emph{left dual} of $X$, together with morphisms
$$
\lev_X\co X^* \otimes X \to \un \quad \text{and} \quad \lcoev_X\co \un \to X \otimes X^*
$$
respectively called \emph{left evaluation} and  \emph{left coevaluation}, such that
\begin{equation*}
(\lev_X \otimes \id_{X^*})(\id_{X^*} \otimes \lcoev_X)=\id_{X^*} \quad \text{and} \quad (\id_X \otimes \lev_X)(\lcoev_X \otimes \id_X)=\id_X.
\end{equation*}
The symmetry of $\sss$ endows $\sss$ with a canonical structure of a pivotal category with \emph{right evaluation} and \emph{right coevaluation} associated to $X \in \Ob(\sss)$ defined respectively by
$$
\rev_X=\lev_X \tau_{X,X^*}\co X \otimes X^* \to \un \quad \text{and} \quad \rcoev_X=\tau_{X,X^*}\lcoev_X\co \un \to X^* \otimes X.
$$
Then $\sss$ is spherical (meaning that the left and right traces induced by the (co)eva\-lua\-tion morphisms coincide) and even ribbon (with trivial twist).
The \emph{trace} of an endomorphism $f \in \End_\sss(X)$ is
$$
\tr(f)=\lev_X(\id_{X^*} \otimes f)\rcoev_X=\rev_X(f \otimes \id_{X^*})\lcoev_X \in \End_\sss(\un).
$$
The \emph{dimension} of an object $X$ is
$$
\dim(X)=\tr(\id_X)=\lev_X\rcoev_X=\rev_X\lcoev_X =\lev_X\tau_{X,X^*}\lcoev_X \in \End_\sss(\un).
$$
The trace is symmetric:   for any morphisms $p\co X \to Y$ and   $q\co Y
\to X$  in~$\sss$,
$$
 \tr(pq)=\tr(qp).
$$
The   trace and dimension are $\otimes$-multiplicative:
$$
 \tr(f\otimes g) =\tr(f) \, \tr(g)   \quad {\text{and}}\quad    \dim(X\otimes Y) = \dim(X) \dim(Y)
$$
for any endomorphisms $f,g$  in~$\sss$ and any object $X,Y$ of $\sss$.

\subsection{Invertible objects}\label{sect-invertible-objects}
An object $I$ of $\sss$ is \emph{invertible} if there is an object $J$ of~$\sss$ such that  $I \otimes J \simeq \un$ (or equivalently $J \otimes I \simeq \un$ because $\sss$ is symmetric). Then~$J$ is unique (up to isomorphism) and is both a left and right dual of $I$. It follows directly from the definition that the monoidal product of two (and so finitely many) invertible objects is an invertible object. Consequently, the isomorphism  classes of invertible objects form a group with  the  multiplication induced by $\otimes$ and the class of $\un$ as the identity element.

Let $I$ be an invertible object of $\sss$. For all $X,Y \in \Ob(\sss)$, the map
$$
\Hom_\sss(X,Y) \to \Hom_\sss(X\otimes I, Y \otimes I), \quad f \mapsto f \otimes \id_I
$$
is bijective, with inverse
$$
\Hom_\sss(X\otimes I, Y \otimes I) \to \Hom_\sss(X,Y), \quad g \mapsto (\id_Y \otimes \psi)(g \otimes \id_J)(\id_X \otimes \psi^{-1})
$$
where $\psi \co I \otimes J \to \un$ is an isomorphism.  In particular,  the map
$$
\End_\sss(\un) \to \End_\sss(I), \quad \alpha \mapsto \alpha\, \id_I
$$
is an isomorphism of monoids (and so $\End_\sss(I)$ is commutative). We denote by~$\norm{\cdot}_I$ the inverse isomorphism: for any $f \in \End_\sss(I)$, there is a unique  $\norm{f}_I \in \End_\sss(\un)$ such that
\begin{equation} \label{eq-bij-invert-ob}
f=\norm{f}_I  \id_I.
\end{equation}

For any invertible object $I$, since $\tau_{I,I} \in  \End_\sss(I \otimes I)$ with $I \otimes I$ invertible, we can consider the scalar
$$
 \sigma_I =\norm{\tau_{I,I}}_{I \otimes I}  \in \End_\sss(\un).
$$
By definition,
\begin{equation} \label{eq-sigma-def}
\tau_{I,I}=\sigma_I \, \id_{I \otimes I}.
\end{equation}
The involutivity of a symmetry implies that
\begin{equation} \label{eq-sigma-invert-ob}
\sigma_I^2=1.
\end{equation}

A left (or a right) dual of an invertible object is invertible.
Assume that $\sss$ is left rigid and endow $\sss$ with the spherical structure canonically induced by the left duality and symmetry (see Section~\ref{sect-duality}).  The uniqueness of a left dual (up to isomorphism) implies that an object $I$ of~$\sss$ is invertible if and only if the left evaluation $\lev_I \co I^* \otimes I \to \un$ is an isomorphism.  Let $I$ be an invertible object of~$\sss$.
Then its dimension $\dim(I) \in \End_\sss(\un)$ is invertible
and
\begin{equation} \label{eq-dimen-invert-ob}
\dim(I)^2=1.
\end{equation}
Indeed, the facts that the dimension of objects is invariant under isomorphisms and is $\otimes$-multiplicative imply that
$$
1=\dim(\un)=\dim(I^* \otimes I)=\dim(I^*)\dim(I).
$$
Then \eqref{eq-dimen-invert-ob} follows from the fact that $\dim(I^*)=\dim(I)$ by the sphericity of $\sss$.
Note that Formulas \eqref{eq-bij-invert-ob} and \eqref{eq-dimen-invert-ob} imply that  for any $f \in \End_\sss(I)$,
$$
\norm{f}_I=\dim(I) \, \tr(f).
$$
Furthermore the scalar $\sigma_I \in \End_\sss(\un)$ is computed by
$$
\sigma_I=\dim(I).
$$
Indeed, since $\sss$ is ribbon with trivial twist, we have $\tr(\tau_{I,I})=\dim(I)$ and so $$\sigma_I=\norm{\tau_{I,I}}_{I \otimes I}=\dim(I \otimes I)\, \tr(\tau_{I,I})=\dim(I)^{3}=\dim(I).$$

\subsection{Example of graded vector spaces}\label{sect-ex-G-vect}
Let $G$ be a group with unit element~$1$. Recall that a vector space $V$ over a field $\kk$ is $G$-graded if it is decomposed as
a direct sum $$U= \bigoplus_{g \in G} V_g$$ of vector subspaces  labeled by  $g\in G$. For $g \in G$, the elements of $V_g$ are said to be homogeneous of degree $g$. For example,   each    $h\in G$  gives rise to a $G$-graded \kt vector space   $\kk_h $ where $(\kk_h)_h=\kk$ and $(\kk_h)_g=0$ for all $g\in G\setminus \{h\}$.
We denote by $G\mti \mathrm{Vect}_\kk$ the category of $G$-graded \kt vector spaces and  \kt linear  grading-preser\-ving homomorphisms. We endow this category with   the   monoidal product    defined on objects   by
$$
U \otimes V=\bigoplus_{g \in G} \, (U \otimes V)_g \quad \text{where} \quad (U \otimes V)_g=\bigoplus_{\substack{h,j \in G \\ \, hj=g}} U_h \otimes_\kk V_j ,
$$
and on morphisms by $f \otimes g=f \otimes_\kk g$. For example,  $\kk_g \otimes \kk_h \simeq \kk_{gh}$ for all $g,h \in G$.
Then  $G\mti \mathrm{Vect}_\kk$ is a monoidal category with unit object $\un=\kk_1$. The group of isomorphism  classes of invertible objects of  $G\mti \mathrm{Vect}_\kk$  is isomorphic to $G$. A representative set of these classes is $\{\kk_g\}_{g \in G}$.

Assume that $G$ is abelian. Any bi-character $\chi\co G \times G \to \kk^*=\kk\setminus \{0\}$ of~$G$ such that $\chi(g,h)\chi(h,g)=1$ for all $g,h \in G$  induces a symmetry for $G\mti \mathrm{Vect}_\kk$ defined~by
$$
\tau_{U,V}(u\otimes v)=\chi(|u|,|v|) \, v\otimes u,
$$
where $U$ and $V$ are $G$-graded \kt vector spaces, $u\in U$ and $v\in V$ are homogeneous elements, and $|x|\in G$ is the degree of a homogeneous element $x$. We denote the resulting symmetric monoidal category by $G\mti \mathrm{Vect}_\kk^\chi$. Note that for any $g \in G$, the scalar \eqref{eq-sigma-def} associated to the invertible object $\kk_g$ of $G\mti \mathrm{Vect}_\kk^\chi$ is computed by
$$
\sigma_{\kk_g}=\chi(g,g).
$$

\subsection{Example of super vector spaces}\label{sect-ex-super-vect}
Let $\kk$ be a field. By Example~\ref{sect-ex-G-vect}, the bi-character $\chi$ of the group $\Z/2\Z=\{0,1\}$ defined by $\chi(g,h)=(-1)^{gh}$ gives rise to the symmetric monoidal category $$\mathrm{SVect}_\kk=(\Z/2\Z)\mti \mathrm{Vect}_\kk^\chi$$ known as the category of super vector spaces over $\kk$. Its symmetry is given by
$$
\tau_{U,V}(u\otimes v)=(-1)^{|u||v|}v\otimes u,
$$
where $U$ and $V$ are objects of $\mathrm{SVect}_\kk$, $u\in U$ and $v\in V$ are homogeneous elements, and $|x|\in \{0,1\}$ is the degree of a homogeneous element $x$. The monoidal unit~$\un$ of $\mathrm{SVect}_\kk$ is defined by $\kk$ in degree $0$ and $0$ in degree~$1$. Denote by $I$ the $(\Z/2\Z)$-graded vector space defined by  0 in degree~0 and $\kk$ in degree 1.
Up to isomorphism, $\un$ and~$I$ are the only invertible objects of $\mathrm{SVect}_\kk$.  Their associated scalar \eqref{eq-sigma-def} are $\sigma_\un=1$ and $\sigma_I=-1$.

\subsection{Hopf algebras}\label{sect-Hopf-alg}
An \emph{algebra} in $\sss$ is an object $A$ of $\sss$ endowed with morphisms $\mu\co A \otimes A \to A$ (the product) and $\eta\co \un \to A$ (the unit) such that
\begin{equation*}
\mu(\mu \otimes \id_A)=\mu(\id_A \otimes \mu) \quad \text{and} \quad \mu(\id_A \otimes \eta)=\id_A=\mu(\eta \otimes \id_A).
\end{equation*}
A \emph{coalgebra} in $\sss$ is an object $C$ of $\sss$ endowed with morphisms $\Delta\co C \to C \otimes C$ (the coproduct) and $\varepsilon\co C \to \un$ (the counit) such that
\begin{equation*}
(\Delta \otimes \id_C)\Delta=(\id_C \otimes \Delta)\Delta \quad \text{and} \quad (\id_C \otimes \varepsilon)\Delta=\id_C=(\varepsilon \otimes \id_C)\Delta.
\end{equation*}
A \emph{bialgebra} in $\sss$ is an object $A$ of~$\sss$ endowed with an algebra structure and a coalgebra structure such that its coproduct $\Delta$ and counit $\varepsilon$ are algebra morphisms (or equivalently, such that its  product
$\mu$  and unit
$\eta$ are coalgebra morphisms), i.e.,
\begin{align*}
\Delta \mu&=(\mu \otimes \mu)(\id_A \otimes \tau_{A,A} \otimes \id_A)(\Delta \otimes \Delta),  & \Delta \eta&=\eta \otimes \eta, \\  \varepsilon \mu &=\varepsilon \otimes \varepsilon, & \varepsilon \eta&=\id_\un.
\end{align*}
An \emph{antipode} for a bialgebra $A$ is a morphism $S\co A \to A$ in $\sss$ such that
\begin{equation*}
\mu(S \otimes \id_A)\Delta=\eta \varepsilon=\mu(\id_A \otimes S)\Delta.
\end{equation*}
If it exists, an antipode is unique and is anti-(co)multipicative, i.e.,
\begin{align*}
  S \mu& =\mu \tau_{A,A}(S \otimes S), & S \eta&=\eta, \\
\Delta S&=(S \otimes S)\tau_{A,A}\Delta, & \varepsilon S &=\varepsilon.
\end{align*}
A \emph{Hopf algebra} in $\sss$ is a bialgebra in $\sss$ which admits an invertible antipode.
A Hopf algebra $A$ in $\sss$ is called \emph{involutory} if its antipode $S$ satisfies $S^2=\id_A$.

Given a Hopf algebra $A$ in $\sss$, we depict its product $\mu$, unit $\eta$,
coproduct $\Delta$, counit $\varepsilon$, and the antipode $S$
as follows:
$$
\psfrag{A}[Bc][Bc]{\scalebox{.9}{$A$}}
\psfrag{n}[Bc][Bc]{\scalebox{.9}{$n$}}
\mu=\,\rsdraw{.35}{.9}{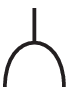}\;, \quad
\eta=\;\rsdraw{.35}{.9}{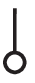}\;, \quad
\Delta=\,\rsdraw{.35}{.9}{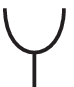}\;, \quad
\varepsilon=\;\rsdraw{.35}{.9}{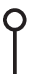}\;, \quad
S=\;\rsdraw{.35}{.9}{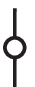}\;.
$$
Here, the underlying color of all the arcs is the object $A$.

For $n \in \N$, we define recursively the morphisms
\begin{equation}\label{es-mun-deln}
\mu_n \co A^{\otimes n} \to A \mand \Delta_n \co A \to A^{\otimes n}
\end{equation}
as follows:
\begin{align*}
&\mu_0=\eta, && \mu_1=\id_A,  && \mu_{n+1}=\mu(\mu_n \otimes \id_A),\\
&\Delta_0=\varepsilon, && \Delta_1=\id_A, && \Delta_{n+1}=(\Delta_n \otimes\id_A)\Delta.
\end{align*}
It follows from the (co)associativity and the (co)unitality of the (co)product that for all $n_1, \dots,n_k \in \N$,
$$
\mu_{n_1+\cdots +n_k}=\mu_k(\mu_{n_1} \otimes \cdots \otimes \mu_{n_k}) \mand
\Delta_{n_1+\cdots +n_k}=(\Delta_{n_1} \otimes \cdots \otimes \Delta_{n_k})\Delta_k.
$$
We will depict the morphisms $\mu_n$ and $\Delta_n$ as
$$
\psfrag{E}[cr][cr]{\scalebox{1.5}[2.6]{$\{$}}
\psfrag{H}[cr][cr]{\scalebox{1.5}[2.6]{$\}$}}
\psfrag{n}[Bc][Bc]{\scalebox{.9}{$n$}}
\mu_n=\;\rsdraw{.6}{.9}{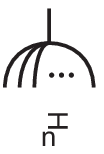}
\mand
\Delta_n=\;\rsdraw{.3}{.9}{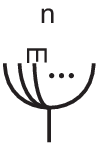} \,.
$$

\subsection{Reverse and dual Hopf algebras}\label{sect-dual-reverse}
Denote by $\sss^\reve$ the category opposite to $\sss$ endowed with the opposite monoidal product.
Then~$\sss^\reve$ is a symmetric monoidal category with  the  symmetry induced by that of~$\sss$. For example, for any objects $A,B,X,Y \in \Ob(\sss^\reve)=\Ob(\sss)$,
$$
\Hom_{\sss^\reve}(A \otimes_{\sss^\reve} \! B, X \otimes_{\sss^\reve} \!  Y)=\Hom_{\sss}(Y \otimes X, B \otimes A).
$$
Let $A$ be a  Hopf algebra in $\sss$, with product $\mu$, unit $\eta$, coproduct $\Delta$, counit $\varepsilon$, and antipode $S$. Then $A$ becomes a Hopf algebra in  $\sss^\reve$  with product $\Delta$, unit~$\varepsilon$, coproduct $\mu$, counit $\eta$, and  antipode $S$. This Hopf algebra in $\sss^\reve$ is denoted by~$A^\reve$ and is called the Hopf algebra \emph{reverse} to $A$. If $A$ is involutory, then so is~$A^\reve$.

Assume that $\sss$ is left rigid (and so spherical, see Section~\ref{sect-duality}). The left (co)eva\-lua\-tions induce a  dual functor $?^* \co \sss^\reve \to \sss$ which is a symmetric monoidal equivalence. Thus, if $A$ is a  Hopf algebra in $\sss$, then the image~$A^*$ of $A^\reve$ under this dual functor is a Hopf algebra in $\sss$ called  the Hopf algebra \emph{dual} to~$A$. Explicitly,  the product and  the coproduct of $A^*$ are
$$
\psfrag{B}[Bl][Bl]{\scalebox{.9}{$A^*$}}
\psfrag{c}[Bc][Bc]{\scalebox{.8}{$\lcoev_A$}}
\psfrag{a}[Bc][Bc]{\scalebox{.8}{$\lev_A$}}
\mu_{A^*}=\, \rsdraw{.45}{.9}{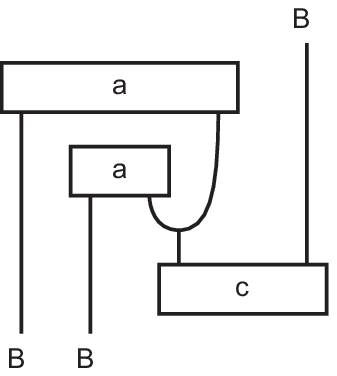}\,, \qquad \Delta_{A^*}= \, \rsdraw{.45}{.9}{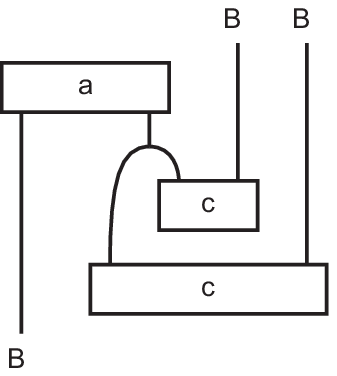}\,,
$$
and the unit, counit, and antipode of $A^*$ are
$$
\psfrag{B}[Bl][Bl]{\scalebox{.9}{$A^*$}}
\psfrag{c}[Bc][Bc]{\scalebox{.8}{$\lcoev_A$}}
\psfrag{a}[Bc][Bc]{\scalebox{.8}{$\lev_A$}}
\eta_{A^*}=\, \rsdraw{.45}{.9}{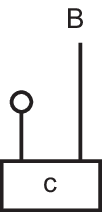}\,, \qquad   \Delta_{A^*}= \, \rsdraw{.45}{.9}{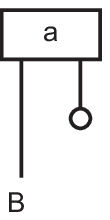}\,,
\qquad  S_{A^*}= \, \rsdraw{.45}{.9}{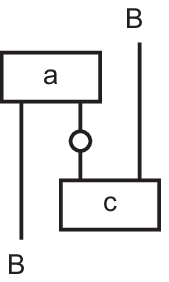} \,.
$$

\subsection{Example of a super Hopf algebra}\label{sect-ex-super-HA1}
Let $\kk$ be a field. Pick a positive integer $n$ not divisible by the characteristic of $\kk$ and an element $\param\in\kk$. Consider the $\kk$-algebra $\Hn$ generated by $t$ and $\ns$ submitted to the relations
$$
t^n=1, \quad \ns t = t \ns, \quad \ns^2=\param (t^2-1).
$$
Note that $\Hn$ is finite-dimensional (over $\kk$) with basis $\{t^k,\, \ns t^k\}_{0 \leq k \leq n-1}$  and $t$ is invertible with inverse $t^{-1}=t^{n-1}$.
By assigning the degrees
$$
|t|=0 \mand |\ns|=1,
$$
we obtain that  $\Hn$ is an algebra in the symmetric monoidal category $\mathrm{SVect}_\kk$ of super vector  \kt spaces (see Section~\ref{sect-ex-super-vect}). Moreover, it becomes a Hopf algebra in $\mathrm{SVect}_\kk$ with coproduct $\Delta$, counit $\varepsilon$, and antipode $S$ given by
\begin{align*}
& \Delta (t)=t\otimes t, && \varepsilon(t)=1, && S(t)=t^{-1}, \\
& \Delta(\ns)=t \otimes \ns +\ns \otimes 1, &&  \varepsilon(\ns)=0, && S(\ns)=-\ns t^{-1}.
\end{align*}
Since $S^2(t)=t$ and $S^2(\ns)=\ns$, the Hopf algebra $\Hn$ in $\mathrm{SVect}_\kk$ is involutory.

\section{Heegaard diagrams and their tensors}\label{sec2}

\subsection{Heegaard diagrams}\label{sect-Heegaard-diag}
By a \emph{circle} on a surface $\Sigma$, we mean the image of an embedding $S^1 \hookrightarrow \Sigma$.

A \emph{Heegaard diagram} is a triple $D=(\Sigma,\up,\low)$ where
\begin{itemize}
\item $\Sigma$ is a closed, connected, and oriented surface of genus $g\geq 1$,
\item $\up$ is a set of $g$ pairwise disjoint circles on $\Sigma$ such that
$\Sigma \setminus \cup_{u \in \up} u$
  is connected,
\item $\low$ is a set of $g$ pairwise disjoint circles on $\Sigma$ such that
$\Sigma \setminus \cup_{l \in \low} l$ is connected,
\item each $u \in \up$ is transverse to each $l \in \low$.
\end{itemize}
The circles in  $\up$ (resp.\@ $\low$) are called the \emph{upper} (resp.\@ \emph{lower}) \emph{circles} of the diagram. The last condition implies that the set $(\cup_{u \in \up} u) \cap (\cup_{l \in \low}l)$ is finite. Its elements are called the \emph{intersection points} of $D$.

For $u \in \up$, we denote by $|u|$ the number of intersection points of $D$ lying in $u$. Likewise, for $l \in \low$, we denote by $|l|$ the number of intersection points of $D$ lying in $l$. Then the total number of intersection points of $D$ is computed by
$$
\sum_{u \in \up} |u|=\sum_{l \in \low} |l|.
$$

\subsection{Heegaard diagrams and 3-manifolds}\label{secct-Heegaard-diag-3man}
To any Heegaard diagram is associated a closed connected oriented 3-manifold $M_D$ obtained as follows.
Consider the 3-manifold $\Sigma \times [0,1]$. Glue a 2-handle along a tubular neighborhood of each circle $l \times \{0\}$ with $l \in \low$. Glue a 2-handle along a tubular neighborhood of each circle $u \times \{1\}$ with $u \in \up$. The result is compact connected 3-manifold whose boundary is a number of spheres. We eliminate these spherical boundary components by gluing in 3-balls to obtain $M_D$. We endow $M_D$ with the orientation extending the product orientation of $\Sigma \times [0,1]$.

Any closed connected oriented 3-manifold may be presented this way by a Heegaard diagram. Moreover, the Reidemeister--Singer theorem asserts that two Heegaard diagrams give rise to homeomorphic oriented 3-manifolds if and only if one can be obtained from the other by a finite sequence of four types of moves (and their inverses), namely:  a homeomorphism of the surface, an isotopy of the diagram (2-point move), a stabilization, and a handle slide (sliding a circle past another), see Section~\ref{sect-proof-main-thm} for details.

\subsection{Oriented, based,  and ordered Heegaard diagrams}\label{sect-opoHD}
A Heegaard diagram is \emph{oriented} if each  lower/upper circle is oriented. A Heegaard diagram is  \emph{based} if each lower/upper circle is endowed with a base point distinct from the intersection points. A Heegaard diagram  $D=(\Sigma,\up,\low)$ is  \emph{ordered} if the sets $\up$ and~$\low$ are (fully) ordered.

The intersection points of an oriented Heegaard diagram $D=(\Sigma,\up,\low)$ have a sign defined as follows. Let $c$ be an intersection point between an upper circle $u$ and a lower circle $l$. Then $c$ is \emph{positive} if $(d_cu,d_c l)$ is a positively-oriented basis of~$\Sigma$ and \emph{negative} otherwise.

Let $D=(\Sigma,\up,\low)$ be an ordered oriented based Heegaard diagram. Then the set $\I$ of intersection points of $D$ inherits a total order from $\up$ as follows. For $u \in \up$, denote by $\I_u$ the set of intersection points of $D$ lying in $u$. Enumerating the elements of $\I_u$ starting from the base point of $u$ and following the orientation of $u$ defines a total order on $\I_u$. Then $\I=\amalg_{u \in \up} \I_u$ is endowed with the lexicographic order induced by the order of $\up$ and that of  $\I_u$. In other words, for $c \in \I_u$ and $c' \in \I_{u'}$ with $u,u'\in\up$, we let $c \leq c'$ if $u < u'$ in $\up$ or if $u=u'$ and $c \leq c'$ in~$\I_u$. Similarly, replacing $\up$ with $\low$, we obtain another total order on the set $\I$ inherited from~$\low$.

\subsection{Tensors associated to Heegaard diagrams using Hopf algebras}\label{sect-asso-tensor-HD}
Let $A$ be a Hopf algebra in a symmetric monoidal category $\sss$. Let $D=(\Sigma,\up,\low)$ be an ordered oriented  based Heegaard diagram. Following Kuperberg \cite{Ku1}, we associate to $D$ a morphism in $\sss$
$$
K_A(D) \co A^{\otimes g} \to A^{\otimes g},
$$
where $g$ is the genus of $\Sigma$, as follows. Denote by $\I$ the set of intersection points of~$D$.
By Section~\ref{sect-opoHD}, the set $\I$ inherits two total orders from $\up$ and $\low$. Denote by $\I_\up$ (respectively $\I_\low$) the set $\I$ endowed with the order induced by $\up$ (respectively~$\low$). Let $N$ be the cardinality of $\I$. Consider the permutation $\sigma \in \mathfrak{S}_N$ such that if an intersection point is the $i$-th element of $\I_\low$, then it is the $\sigma(i)$-th element of~$\I_\up$. By Section~\ref{sect-repres-Sn}, we derive from~$\sigma$ and the symmetry of $\sss$ an automorphism
$$
P_\sigma \co A^{\otimes N} \to A^{\otimes N}.
$$
Set
$$
S_\low=\bigotimes_{c \in \I_\low} S^{\kappa_c}  \co A^{\otimes N} \to A^{\otimes N}  \mand
S_\up=\bigotimes_{c \in \I_\up} S^{\kappa_c}  \co A^{\otimes N} \to A^{\otimes N},
$$
where $S$ is the antipode of $A$ and
$$
\kappa_c=
\begin{cases}
0 & \text{if $c$ is positive,} \\ 1 & \text{if $c$ is negative.}
\end{cases}
$$
Here (and in what follows), a monoidal product indexed by a fully ordered set is taken following the order of the set of indices (as in Section~\ref{sect-conventions}). Note that the naturality of symmetry and the definition of $\sigma$ imply that
$$
P_\sigma S_\low= S_\up P_\sigma.
$$
For $l \in \low$ and $u \in \up$, consider the integers $|u|$ and $|l|$ defined in Section~\ref{sect-Heegaard-diag} and the morphisms $\Delta_{|l|}$ and $\mu_{|u|}$ defined in Section~\ref{sect-Hopf-alg}. Recall that $N$ is computed by
$$
N=\sum_{u \in \up} |u|=\sum_{l \in \low} |l|.
$$
Set
$$
\Delta_\low=\bigotimes_{l \in \low} \Delta_{|l|}  \co A^{\otimes g} \to A^{\otimes N}  \mand
\mu_{\hspace{1pt}\up}=\bigotimes_{u \in \up} \mu_{|u|} \co A^{\otimes N} \to A^{\otimes g}.
$$
Finally, define
\begin{equation}\label{eq-KAD}
K_A(D)=\mu_{\hspace{1pt}\up} P_\sigma S_\low  \Delta_\low =  \mu_{\hspace{1pt}\up} S_\up P_\sigma  \Delta_\low \co A^{\otimes g} \to A^{\otimes g}.
\end{equation}

\subsection{Examples}\label{sect-ex-KAD}
1) Let $\T=\R^2/\Z^2$ be the torus with orientation induced by the counterclockwise orientation of $\R^2$.  Consider the following circles on $\T$:
$$
u= (1/2,0)+\R(0,1) + \Z^2 \mand l=\R(0,1) + \Z^2.
$$
Then $D_0=(\T,\{u\},\{l\})$ is a Heegaard diagram which represents $S^1 \times S^2$. Since $D_0$ has no intersection points, we obtain that
$$
K_A(D_0)=\mu_0 \,\Delta_0= \;\rsdraw{.4}{.9}{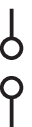}
$$
for any choice for the orientations and base points of the circles.

2) Consider as above the oriented torus $\T=\R^2/\Z^2$. Pick a positive integer $p$ and define the following circles on $\T$:
$$
u=\R(1,1/p) + \Z^2 \mand l=\R(0,1) + \Z^2.
$$
Then $D_p=(\T,\{u\},\{l\})$ is a Heegaard diagram which represents the lens space $L(p,1)$. Clearly, it is ordered (since the sets of upper/lower circles are singletons). We provide the circles $u$ an $l$ with the orientation induced by the standard orientation of~$\R$. We endow $u$ with the base point $(1/2,1/(2p))$ and $l$ with the base point $(0,1/2p)$. Then $D_p$ is an ordered oriented  based  Heegaard diagram with $p$ intersection points which are all positive The permutation associated to $D_p$ is the identity and so we have $P_{(1 2 \cdots p)}=\id_{A^{\otimes p}}$. Then
$$
K_A(D)=\mu_p \,\Delta_p =\;\rsdraw{.4}{.9}{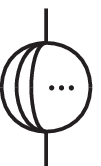} \;.
$$

3) Let us consider the following oriented based Heegaard diagram $D=(\Sigma,\up,\low)$ of genus 2  used by Poincar\'e to define  his celebrated homology 3-sphere  (see \cite{Po}):
$$
\psfrag{1}[Bc][Bc]{\scalebox{.9}{$1$}}
\psfrag{2}[Bc][Bc]{\scalebox{.9}{$2$}}
\psfrag{3}[Bc][Bc]{\scalebox{.9}{$3$}}
\psfrag{4}[Bc][Bc]{\scalebox{.9}{$4$}}
\psfrag{5}[Bc][Bc]{\scalebox{.9}{$5$}}
\psfrag{6}[Bc][Bc]{\scalebox{.9}{$6$}}
\psfrag{7}[Bc][Bc]{\scalebox{.9}{$7$}}
\psfrag{a}[Bc][Bc]{\scalebox{.9}{$a$}}
\psfrag{b}[Bc][Bc]{\scalebox{.9}{$b$}}
\psfrag{c}[Bc][Bc]{\scalebox{.9}{$c$}}
\psfrag{d}[Bc][Bc]{\scalebox{.9}{$d$}}
\psfrag{e}[Bc][Bc]{\scalebox{.9}{$e$}}
\psfrag{l}[Bc][Bc]{\scalebox{.9}{$l_1$}}
\psfrag{k}[Bc][Bc]{\scalebox{.9}{$l_1$}}
\psfrag{x}[Bc][Bc]{\scalebox{.9}{$l_2$}}
\psfrag{y}[Bc][Bc]{\scalebox{.9}{$l_2$}}
\psfrag{u}[Bc][Bc]{\scalebox{.9}{$u_1$}}
\psfrag{v}[Bc][Bc]{\scalebox{.9}{$u_2$}}
\psfrag{S}[Bc][Bc]{\scalebox{.9}{$\Sigma$}}
D= \;\rsdraw{.45}{.9}{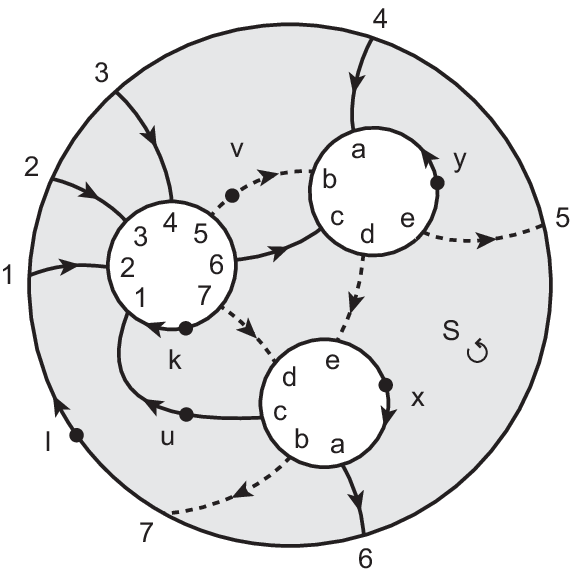} \;.
$$
This pictures represents the surface $\Sigma$ cut along the lower circles $l_1$ and $l_2$. The arrows and the dots represent the orientations and  base points of the circles. We order the sets $\up=\{u_1,u_2\}$ and $\low=\{l_1,l_2\}$ by setting $u_1<u_2$ and $l_1<l_2$. Then~$D$ is an ordered oriented based Heegaard diagram. We have:
\begin{align*}
& l_1\cap (u_1 \cup u_2)=\{1<2<3<4<5<6<7\},\\
& l_2\cap (u_1 \cup u_2)=\{a<b<c<d<e\}, \\
& u_1\cap (l_1 \cup l_2)=\{1<2<3<4<a<6<c\}, \\
& u_2\cap (l_1 \cup l_2)=\{b<7<d<e<5\}.
\end{align*}
Thus the set of intersection points ordered by $\low$ is
$$\{1<2<3<4<5<6<7<a<b<c<d<e\}$$
and the set of intersection points ordered by $\up$ is
$$\{1<2<3<4<a<6<c<b<7<d<e<5\}.$$
Consequently the permutation $\sigma \in \mathfrak{S}_{12}$ associated to $D$ is the cycle
$$
\sigma=(5,12,11,10,7,9,8).
$$
Since the positive intersection points are $1,2,3,4,d,e$ and the negative ones are $5,6,7,a,b,c$, we obtain that
$$
S_\low=\id_{A^{\otimes 4}} \otimes S^{\otimes 6} \otimes \id_{A^{\otimes 2}}
 \mand S_\up=\id_{A^{\otimes 4}} \otimes S^{\otimes 5} \otimes \id_{A^{\otimes 2}} \otimes S.
$$
Then
$$
K_A(D)=(\mu_7 \otimes \mu_5) P_\sigma  S_\low (\Delta_7 \otimes \Delta_5)=(\mu_7 \otimes \mu_5) S_\up P_\sigma  (\Delta_7 \otimes \Delta_5)  \co A^{\otimes 2} \to A^{\otimes 2}.
$$
Graphically:
$$
K_A(D)=\; \rsdraw{.45}{.9}{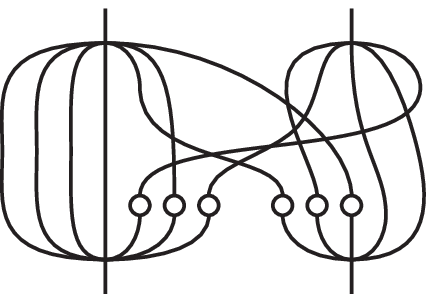} \;.
$$

\section{Good pairs and 3-manifolds invariants}\label{sect-inv-from-gp}\label{sec3}

In this section, $A=(A,\mu,\eta,\Delta,\varepsilon,S)$ is an involutory Hopf algebra in a symmetric monoidal category $\sss$ with symmetry $\tau$. We derive a topological invariant of closed oriented 3-manifolds by evaluating the morphisms \eqref{eq-KAD} using so-called good pairs of~$A$ that we first define.

\subsection{Good pairs}\label{sect-def-good-pair}
A \emph{good pair} for $A$ is a pair $(\phi\co A \to I,\, \Omega\co I \to A)$ of morphisms in $\sss$, where $I$ is  an invertible object of $\sss$, satisfying the following four axioms:
\begin{enumerate}
\labeli
\item[(GP1)] $\phi\Omega=\id_I\in \End_\sss(I)$.
Graphically:
$$
\rsdraw{.4}{.9}{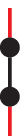} \,=\; \rsdraw{.4}{.9}{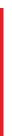}\;.
$$
Here and in what follows, a black-colored strand is colored by the object~$A$, a red-colored strand is colored by the object~$I$, and we depict the morphisms $\phi$ and $\Omega$ as
$$
\phi=\,\rsdraw{.4}{.9}{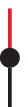} \mand \Omega=\,\rsdraw{.4}{.9}{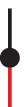}\;\,.
$$
\item[(GP2)] Setting $H=(\mu \otimes\id_A)(\id_A \otimes\Delta) \co A \otimes A \to A \otimes A$, we require that
$$(\phi \otimes \phi)H=\phi \otimes \phi \mand  H(\Omega \otimes \Omega)=\Omega \otimes \Omega.$$
Graphically:
$$
\rsdraw{.4}{.9}{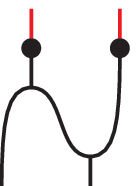} \,=\;\, \rsdraw{.4}{.9}{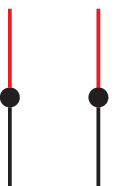} \mand \rsdraw{.4}{.9}{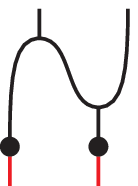} \,=\;\, \rsdraw{.4}{.9}{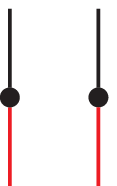}\;\,.
$$
\item[(GP3)] Setting $\nu_{(\phi,\Omega)}=\norm{\phi S \Omega}_I \in \End_\sss(\un)$, we require that
$$
\phi S=\nu_{(\phi,\Omega)}\, \phi \mand S \Omega=\nu_{(\phi,\Omega)}\,\Omega.
$$
Graphically:
$$
\rsdraw{.45}{.9}{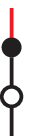} \,=\,\nu_{(\phi,\Omega)}\;\, \rsdraw{.45}{.9}{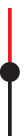} \mand
\rsdraw{.45}{.9}{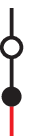} \,=\,\nu_{(\phi,\Omega)}\;\, \rsdraw{.45}{.9}{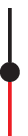} \;\,.
$$
\item[(GP4)]
There is a morphism $f \co A \to A$ in $\sss$ such that
$$
\phi \mu \tau_{A,A}=\phi \mu (f \otimes \id_A) \mand \Delta \Omega=(f  \otimes \id_A)\Delta \Omega.
$$
Graphically:
$$
\psfrag{h}[Bc][Bc]{\scalebox{.9}{$f$}}
\rsdraw{.4}{.9}{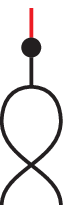} \,=\, \rsdraw{.4}{.9}{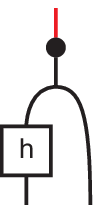} \; \mand \;
\rsdraw{.4}{.9}{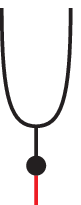} \,=\, \rsdraw{.4}{.9}{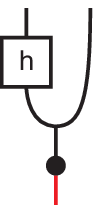}\;.
$$
\item[(GP5)]
There is a morphism  $h \co A \to A$ in $\sss$ such that
$$
\tau_{A,A} \Delta  \Omega = (\id_A \otimes h)\Delta \Omega \mand  \phi \mu=\phi \mu (\id_A \otimes h).
$$
Graphically:
$$
\psfrag{h}[Bc][Bc]{\scalebox{.9}{$h$}}
\rsdraw{.4}{.9}{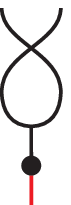} \,=\, \rsdraw{.4}{.9}{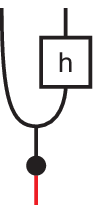}\; \mand \;
\rsdraw{.4}{.9}{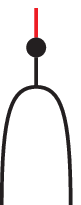} \,=\, \rsdraw{.4}{.9}{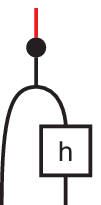}\;.
$$
\end{enumerate}
\medskip

In the next lemma, we note some properties of a good pair.
\begin{lemma}\label{lem-gp}
Let  $(\phi,\Omega)$ be a good pair for $A$. Then:
\begin{enumerate}
\labela
\item $\nu_{(\phi,\Omega)}^2=1$.
\item Setting $H'=(\id_A \otimes \mu)(\Delta \otimes \id_A) \co A \otimes A \to A \otimes A$, we have:
$$(\phi \otimes \phi)H'=\phi \otimes \phi \mand  H'(\Omega \otimes \Omega)=\Omega \otimes \Omega.$$
Graphically:
$$
\rsdraw{.4}{.9}{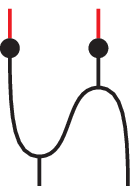} \,=\;\, \rsdraw{.4}{.9}{gp-2b} \mand \rsdraw{.4}{.9}{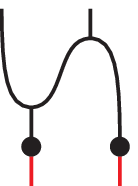} \,=\;\, \rsdraw{.4}{.9}{gp-2d}\;\,.
$$
\item For any $\epsilon \in\{0,1\}$,
$$
(\id_A \otimes fS^\epsilon \otimes \id_A)\Delta_3 \Omega=(\id_A \otimes S^\epsilon \otimes \id_A)\Delta_3 \Omega
$$
where $f$ is as in Axiom {\rm (GP4)}. Graphically:
$$
\psfrag{h}[Bc][Bc]{\scalebox{.9}{$f$}}
\rsdraw{.45}{.9}{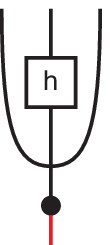} \;=\; \rsdraw{.45}{.9}{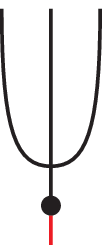}
\, \mand \;
\rsdraw{.45}{.9}{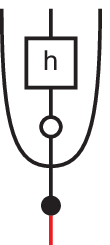} \;=\; \rsdraw{.45}{.9}{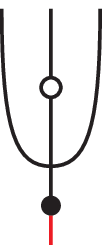}\,.
$$
\item For any $\epsilon \in\{0,1\}$,
$$
\phi \mu_3(\id_A \otimes S^\epsilon h \otimes \id_A)=\phi \mu_3(\id_A \otimes S^\epsilon \otimes \id_A)
$$
where $h$ is as in Axiom {\rm (GP5)}. Graphically:
$$
\psfrag{h}[Bc][Bc]{\scalebox{.9}{$h$}}
\rsdraw{.4}{.9}{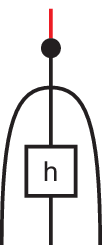} \;=\; \rsdraw{.4}{.9}{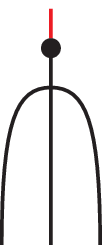}
\, \mand \;
\rsdraw{.4}{.9}{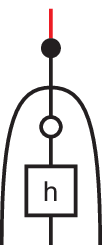} \;=\; \rsdraw{.4}{.9}{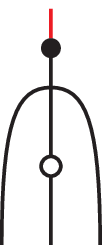}\,.
$$
\end{enumerate}
\end{lemma}
\begin{proof}
Let us prove Part (a). We have:
$$
1 \overset{(i)}{=} \norm{\phi \Omega}_I \overset{(ii)}{=} \norm{\phi S^2 \Omega}_I \overset{(iii)}{=} \nu_{(\phi,\Omega)}  \norm{\phi S \Omega}_I \overset{(iv)}{=} \nu_{(\phi,\Omega)}^2 \norm{\phi \Omega}_I  \overset{(v)}{=} \nu_{(\phi,\Omega)}^2.
$$
Here $(i)$ and $(v)$ follow from (GP1), $(ii)$ from the involutivity of $S$, and $(iii)$ and~$(iv)$ from (GP3).
Part (b) follows from Axioms (GP2)-(GP3), Part (a), and the fact that
$$
H'=\tau_{A,A}(S \otimes S) H (S \otimes S) \tau_{A,A}.
$$
Let us prove Part (c) for $\epsilon=0$. We have:
\begin{gather*}
\psfrag{h}[Bc][Bc]{\scalebox{.9}{$h$}}
\rsdraw{.45}{.9}{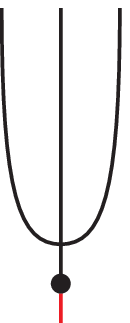}
\;\;\overset{(i)}{=}\; \rsdraw{.45}{.9}{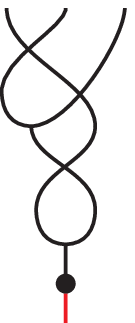}
\;\overset{(ii)}{=}\;\, \rsdraw{.45}{.9}{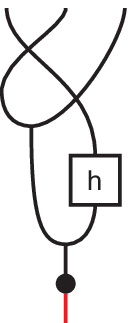}
\;\;\,\overset{(iii)}{=}\;\;\, \rsdraw{.45}{.9}{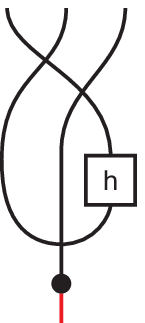}\;\,.
\end{gather*}
Here $(i)$ and $(iii)$ follow from the coassociativity of $\Delta$ and the naturality of $\tau$, and $(ii)$ from (GP5). Then
\begin{gather*}
\psfrag{h}[Bc][Bc]{\scalebox{.9}{$h$}}
\psfrag{k}[Bc][Bc]{\scalebox{.9}{$f$}}
\rsdraw{.45}{.9}{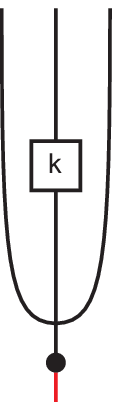}
\;\;\,\overset{(i)}{=}\;\;\, \rsdraw{.45}{.9}{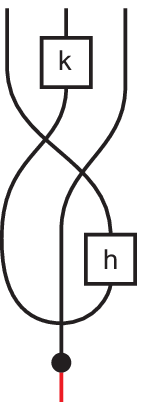}
\;\;\,\overset{(ii)}{=}\;\;\, \rsdraw{.45}{.9}{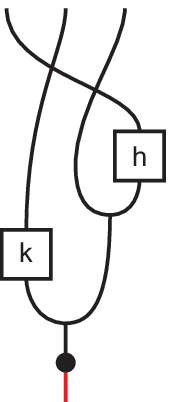} \\[.6em]
\psfrag{h}[Bc][Bc]{\scalebox{.9}{$h$}}
\psfrag{k}[Bc][Bc]{\scalebox{.9}{$f$}}
\overset{(iii)}{=}\; \rsdraw{.45}{.9}{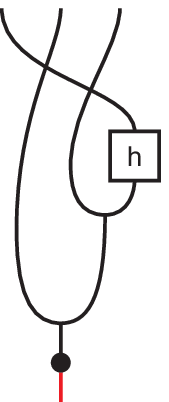}
\;\;\,\overset{(iv)}{=}\;\;\, \rsdraw{.45}{.9}{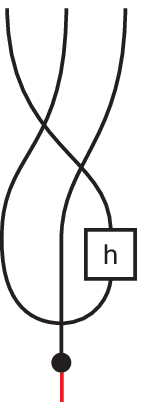}
\;\;\overset{(v)}{=}\;\;\; \rsdraw{.45}{.9}{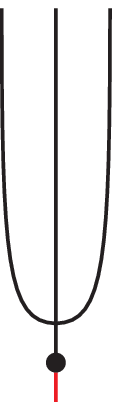}\;\,.
\end{gather*}
Here $(i)$ and $(v)$ follow from the latter computation, $(ii)$ from the naturality of $\tau$ and the coassociativity of $\Delta$,
$(iii)$ from (GP4), and $(iv)$ from the coassociativity of~$\Delta$.

Let us prove Part (c) for $\epsilon=1$. It follows from the anti-comultiplicativity and involutivity of the antipode that
\begin{gather*}
\rsdraw{.46}{.9}{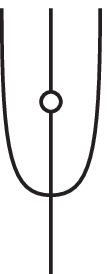}
\;\;=\;\;\, \rsdraw{.46}{.9}{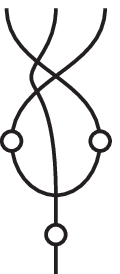}\;\,.
\end{gather*}
Then
\begin{gather*}
\psfrag{h}[Bc][Bc]{\scalebox{.9}{$f$}}
\rsdraw{.45}{.9}{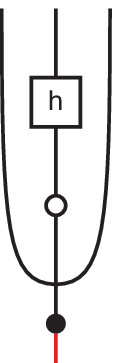}
\;\;\, \overset{(i)}{=}\;\;\, \rsdraw{.45}{.9}{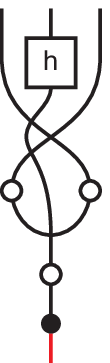}
\;\;\, \overset{(ii)}{=} \;\, \nu_{(\phi,\Omega)} \;\;\; \rsdraw{.45}{.9}{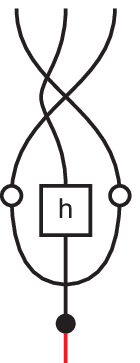}\\[.6em]
\overset{(iii)}{=} \;\, \nu_{(\phi,\Omega)} \;\;\; \rsdraw{.45}{.9}{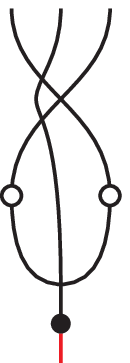}
\;\;\, \overset{(iv)}{=}\;\;\, \rsdraw{.45}{.9}{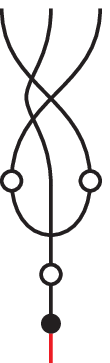}
\;\;\, \overset{(v)}{=}\;\;\; \rsdraw{.45}{.9}{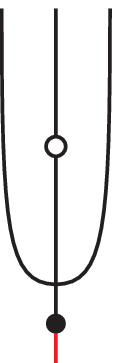}\;\,.
\end{gather*}
Here $(i)$ and $(v)$ follow from the latter equality, $(ii)$ from (GP3) and the naturality of~$\tau$, $(iii)$ from Part (c) for $\epsilon=0$, and $(iv)$ from (GP3). This completes the proof of Part~(c). Part (d) is proved similarly.
\end{proof}

\subsection{Remark}\label{rem-revese-gp}
Let $\phi\co A \to I$ and $\Omega\co I \to A$ be morphisms in $\sss$, where $I$ is  an invertible object of $\sss$.
Consider the reverse Hopf algebra $A^\reve$ in $\sss^\reve$ (see Section~\ref{sect-dual-reverse}). Then it follows directly from the definitions (by just rotating the diagrams by~$180^\circ$) that  $(\phi,\Omega)$ is a good pair of $A$ if and only if $(\Omega,\phi)$ is a good pair of~$A^\reve$.

\subsection{Invariants of 3-manifolds from good pairs}\label{sect-main-result}
Let $(\phi\co A \to I,\, \Omega\co I \to A)$ be
a good pair of $A$. We associate to any oriented closed connected 3-manifold $M$ a scalar
$$
K_{(\phi,\Omega)}(M) \in \End_\sss(\un)
$$
as follows. Pick a Heegaard diagram $D=(\Sigma,\up,\low)$ of $M$.  Endow each upper and lower circle with an orientation and a base point, and pick (full) orders for the sets~$\up$ and $\low$. Then $D$ is an ordered oriented based Heegaard diagram.  According to the recipe described in Section~\ref{sect-asso-tensor-HD}, the Hopf algebra $A$  associates to $D$ the morphism
$$
K_A(D) \co A^{\otimes g} \to A^{\otimes g}
$$
in $\sss$, where $g$ is the genus of the surface $\Sigma$. Then
$$
\phi^{\otimes g} \circ K_A(D) \circ \Omega^{\otimes g}\in \End_\sss(I^{\otimes g}).
$$
Since $I^{\otimes g}$ is invertible (because $I$ is invertible), Section~\ref{sect-invertible-objects} yields a scalar
$$
\norm{\phi^{\otimes g} \circ K_A(D) \circ \Omega^{\otimes g}}_{I^{\otimes g}} \in \End_\sss(\un).
$$
Recall the scalars $\sigma_I$ and $\nu_{(\phi,\Omega)}$ from Sections~\ref{sect-invertible-objects} and~\ref{sect-def-good-pair}, respectively, and let
$$
\gamma_{(\phi,\Omega)}=
\begin{cases}
1 & \text{if $\sigma_I=\nu_{(\phi,\Omega)}=1$,} \\ 2 & \text{otherwise.}
\end{cases}
$$
Finally, set
\begin{equation*}
\inv_{(\phi,\Omega)}(M) =\norm{\phi^{\otimes g} \circ K_A(D) \circ \Omega^{\otimes g}}_{I^{\otimes g}}^{\gamma_{(\phi,\Omega)}}
\in \End_\sss(\un).
\end{equation*}

\begin{theorem}\label{main-thm}
$\inv_{(\phi,\Omega)}(M)$ is a topological invariant of $M$.
\end{theorem}

We prove Theorem~\ref*{main-thm} in Section~\ref{sect-proof-main-thm}.

It follows from Section~\ref{sect-ex-KAD} that
$$
\inv_{(\phi,\Omega)}(S^1 \times S^2)= \norm{\alpha_0}_I^{\gamma_{(\phi,\Omega)}} \mand
\inv_{(\phi,\Omega)}(L(p,1))= \norm{\alpha_p}_I^{\gamma_{(\phi,\Omega)}}
$$
where $p \geq 1$ and
$$
\alpha_0=\; \rsdraw{.45}{.9}{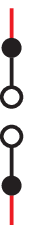}=\phi \eta \, \varepsilon\Omega =\phi \mu_0  \Delta_0 \Omega, \qquad
\alpha_p=\; \rsdraw{.45}{.9}{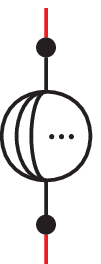}=\phi \mu_p  \Delta_p \Omega.
$$
In particular, since $S^3=L(1,1)$ and  $\norm{\alpha_1}_I=\norm{\phi \mu_1  \Delta_1 \Omega}_I=\norm{\phi\Omega}_I=\norm{\id_I}_I=1$, we obtain that
$$
\inv_{(\phi,\Omega)}(S^3)=1.
$$
Also, denoting by $\mathbb{PO}$ the Poincar\'e homology sphere, we have:
$$
\inv_{(\phi,\Omega)}(\mathbb{PO})=\norm{\beta}_{I \otimes I}^{\gamma_{(\phi,\Omega)}}
\quad \text{where} \quad
\beta=\; \rsdraw{.45}{.9}{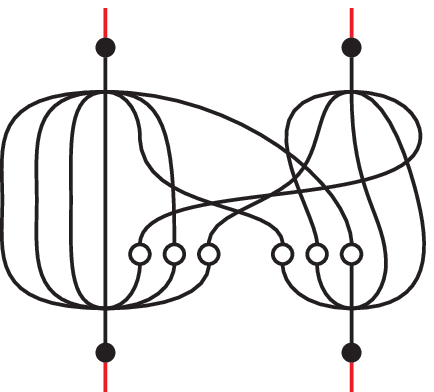} \;.
$$

\subsection{Remarks}
1) Let us consider the case where $\sss$ is the category $\mathrm{vect}_\kk$ of finite-dimensional vector spaces over a field $\kk$, that is, where $A$ is a finite-dimensional involutory Hopf algebra over $\kk$. Assume that $\dim(A) 1_\kk$ is invertible in $\kk$. Let $M$ be an oriented closed  connected 3-manifold. Denote by $\mathrm{Ku}_A(M)$ the Kuperberg invariant of $M$ derived from $A$, see \cite{Ku1}. Then $\mathrm{Ku}_A(M)$ is recovered from Theorem~\ref{main-thm} as follows. The involutivity of the antipode of $A$ and the invertibility of  $\dim(A) 1_\kk$ imply that $A$ has a two-sided integral $\Lambda \in \Hom_\kk( \kk, A)=A$ and a two-sided cointegral $\lambda \in \Hom_\kk(A, \kk)=A^*$ such that $\lambda(\Lambda)=1_\kk$. Then $(\Lambda,\lambda)$ is a good pair for $A$ such that $\gamma_{(\Lambda,\lambda)}=1$ (because $\sigma_\kk=\nu_{(\Lambda,\lambda)}=1$) and
$$
\mathrm{Ku}_A(M)=\inv_{(\Lambda,\lambda)}(M).
$$

2) Let $\mathcal{T}$ be a symmetric monoidal category and $F\co \sss \to \mathcal{T}$ be a symmetric monoidal equivalence.
In particular,  $F(A)$ is an  involutory Hopf algebra in $\mathcal{T}$. It follows from the definitions that if $(\phi,\Omega)$ is a good pair of~$A$, then $(F(\phi),F(\Omega))$ is a good pair of~$F(A)$ and, for any oriented closed 3-manifold $M$,
$$
\inv_{(F(\phi),F(\Omega))}(M)=\inv_{(\phi,\Omega)}(M).
$$

3) Let  $(\phi,\Omega)$ be a good pair of~$A$. Consider the reverse Hopf algebra $A^\reve$ in~$\sss^\reve$ (see Section~\ref{sect-dual-reverse}). By Remark~\ref{rem-revese-gp}, $(\Omega,\phi)$ is a good pair of $A^\reve$.
Then, for any oriented closed 3-manifold $M$,
$$
\inv_{(\Omega,\phi)}(M)=\inv_{(\phi,\Omega)}(M).
$$
This  follows from the fact that if $D=(\Sigma,\up,\low)$ is an  ordered oriented  based Heegaard diagram, then
$$
K_{A^\reve}(D)=K_A(D^\reve)
$$
where $D^\reve=(-\Sigma,\low^\opp,\up^\opp)$  is the ordered oriented  based  Heegaard diagram obtained from $D$
by reversing the orientation of $\Sigma$, reversing the orientations of the upper/lower circles while keeping their base points intact, exchanging the lower and upper circles, and taking the opposite order for $\up$ and $\low$.

4) Assume that $\sss$ is left rigid and let $(\phi,\Omega)$ be
a good pair of~$A$. Consider the dual Hopf algebra $A^*$ in~$\sss$ (see Section~\ref{sect-dual-reverse}).
The latter two remarks, together with the fact that the dual functor $?^* \co \sss^\reve \to \sss$ is a symmetric monoidal equivalence, imply that $(\Omega^*,\phi^*)$ is a good pair of $A^*$ and, for any oriented closed 3-manifold  $M$,
$$
\inv_{(\Omega^*,\phi^*)}(M)=\inv_{(\phi,\Omega)}(M).
$$

5) As remarked by Andrew Casson (see Kuperberg~\cite[Proposition 2.5]{Ku2}), the sets of upper and lower circles of an oriented Heegaard diagram have canonical sign-orderings\footnote{A \emph{sign-ordering} of a finite set is an orbit of the alternating group acting on the set of full orderings of the set.}. Here canonical means that the sign-ordering is preserved by handle slide and reversed if the orientation of a Heegaard circle is reversed. If instead of picking an arbitrary ordering for the sets of upper and lower circles we take this sign ordering, then the definition of $\inv_{(\phi,\Omega)}(M)$ may be slightly improved  by redefining $\gamma_{(\phi,\Omega)}$ as follows:
$$
\gamma_{(\phi,\Omega)}=
\begin{cases}
1 & \text{if $\nu_{(\phi,\Omega)}=1$,} \\ 2 & \text{otherwise.}
\end{cases}
$$
Note that this definition of $\gamma_{(\phi,\Omega)}$ differs from that given in Section~\ref{sect-main-result} only in the case $\sigma_I\neq 1=\nu_{(\phi,\Omega)}$.
The definition of $\gamma_{(\phi,\Omega)}$ given in Section~\ref{sect-main-result} avoids the computation of the canonical sign-orderings.

\subsection{Example}\label{sect-ex-calcul-inv}
Let $\kk$ be a field, $n$ be a positive integer not divisible by the characteristic of $\kk$, and $\param\in\kk$.
Consider the Hopf algebra $\Hn$ in the category $\mathrm{SVect}_\kk$ of super \kt vector spaces defined in Section~\ref{sect-ex-super-HA1}. By Section~\ref{sect-ex-super-vect}, the $(\Z/2\Z)$-graded vector space $I$ defined by  0 in degree~0 and $\kk$ in degree 1 is invertible and $\sigma_I=-1$.  Define morphisms
$$
\phi \co \Hn \to I \mand \Omega \co I \to \Hn
$$
by setting, for $0 \leq k \leq n-1$,
$$
\phi(t^k)=0, \quad \phi(\ns t^k)=\frac{1}{n} 1_\kk, \mand  \Omega(1_\kk)=\ns (1+t+ \cdots + t^{n-1}).
$$
Then $(\phi,\Omega)$ is a good pair for $\Hn$ (see  Example~\ref{sect-ex-gp-from-integs} below) for which $\nu_{(\phi,\Omega)}=-1$ and so $\gamma_{(\phi,\Omega)}=2$. Using this good pair, we obtain:
\begin{align*}
& \inv_{(\phi,\Omega)}(S^3)=1,  && \inv_{(\phi,\Omega)}(S^1 \times S^2)=0, \\
& \inv_{(\phi,\Omega)}(L(p,1))=p^2, && \inv_{(\phi,\Omega)}(\mathbb{PO})=  1
\end{align*}
where $p \geq 1$ and $\mathbb{PO}$ is the Poincar\'e homology sphere.

\subsection{Proof of Theorem~\ref*{main-thm}}\label{sect-proof-main-thm}
Let $D=(\Sigma,\up,\low)$ be a genus $g$ Heegaard diagram of an oriented closed connected 3-manifold $M$. As in Section~\ref{sect-main-result}, endow each upper and lower circle with an orientation and a base point, and pick (full) orders for the sets $\up$ and $\low$, so that $D$ becomes an ordered oriented based  Heegaard diagram. Recall that
$$
\inv_{(\phi,\Omega)}(M) =\norm{\Upsilon_D}_{I^{\otimes g}}^{\gamma_{(\phi,\Omega)}}
\quad \text{where} \quad
\Upsilon_D=\phi^{\otimes g} \circ K_A(D) \circ \Omega^{\otimes g}.
$$

First, suppose that we change the orientation of an upper circle $u_0 \in \up$ into its opposite. This transforms $D$ into an ordered oriented based Heegaard diagram~$D'$. Using the notation of Sections~\ref{sect-opoHD} and~\ref{sect-asso-tensor-HD}, we have:
$$
K_A(D)= \left (\bigotimes_{u \in \up} \mu_{|u|} S_u \right )  P_\sigma  \Delta_\low
\quad \text{where} \quad
S_u=\bigotimes_{c \in \I_u} S^{\kappa_c}.
$$
Note that changing the orientation of $u_0$ transforms the order of $\I_{u_0}$ into its opposite and transforms the positive/negative intersection points in $u_0$ into negative/positive ones.  Also, the involutivity of $S$ implies that for $\kappa \in \{0,1\}$, $$S^{1-\kappa}=SS^{-\kappa}=SS^{\kappa}.$$
Thus $K_A(D')$ is obtained from $K_A(D)$ by replacing $\mu_{|u_0|}S_{u_0}$ with
$$
\mu_{|u_0|} P_\omega S^{\otimes |u_0|} S_{u_0}
$$
where the permutation $\omega\in\mathfrak{S}_{|u_0|}$ is defined by $\omega(i)=|u_0|-i+1$. Now the anti-multiplicativity of the antipode $S$ implies that
$$
\mu_{|u_0|} P_\omega S^{\otimes |u_0|} S_{u_0}=S \mu_{|u_0|} S_{u_0}.
$$
Thus, by Axiom (GP3) of a good pair, we obtain that
$$
\Upsilon_{D'}=\nu_{(\phi,\Omega)}\, \Upsilon_D.
$$
Now $\nu_{(\phi,\Omega)}^{\gamma_{(\phi,\Omega)}}=1$ by the definition of $\gamma_{(\phi,\Omega)}$ and Lemma~\ref{lem-gp}(a). Consequently,
$$
\norm{\Upsilon_{D'}}_{I^{\otimes g}}^{\gamma_{(\phi,\Omega)}}
=\nu_{(\phi,\Omega)}^{\gamma_{(\phi,\Omega)}}\, \norm{\Upsilon_D}_{I^{\otimes g}}^{\gamma_{(\phi,\Omega)}}
=\norm{\Upsilon_D}_{I^{\otimes g}}^{\gamma_{(\phi,\Omega)}}.
$$
This proves the invariance under the choice of orientation for the upper circles. The  invariance under the choice of orientation for the lower circles is proved similarly.

Second, let us prove the invariance under the change of the base point of an upper circle $u_0\in\up$. Clearly, $K_A(D)$ (and so $\Upsilon_D$) remains unchanged if we choose the base point within the same connected component of $u_0\setminus \I_{u_0}$. Thus we only need to consider the case where $|u_0| \geq 2$ and the base point of $u_0$ is changed by picking it in the previous connected component of $u_0\setminus \I_{u_0}$ (the set of connected components of $u_0\setminus \I_{u_0}$ inherits a cyclic order from the orientation of $u_0$). This transforms $D$ into an ordered oriented based Heegaard diagram $D'$.
The change of the base point modifies cyclically the order $\I_{u_0}$ (the last element becoming the first one) and keeps the positivity/negativity of the intersection points intact. Thus $K_A(D')$ is obtained from
$$
K_A(D)= \left (\bigotimes_{u \in \up} \mu_{|u|} \right )  P_\sigma S_\low \Delta_\low
$$
by replacing $\mu_{|u_0|}$ with
$$
\mu_{|u_0|} \tau_{A,A^{\otimes(|u_0|-1)}}
.$$
Now, using the morphism $f$ provided by Axiom (GP4) of a good pair, we have:
\begin{align*}
\phi\mu_{|u_0|} \tau_{A,A^{\otimes(|u_0|-1)}}
 & \overset{(i)}{=} \phi \mu \tau_{A,A}(\id_A  \otimes \mu_{(|u_0|-1)}) \\
 & \overset{(ii)}{=} \phi \mu (f  \otimes \mu_{(|u_0|-1)}) \\
 & \overset{(iii)}{=}\phi\mu_{|u_0|}(f \otimes \id_{A^{\otimes(|u_0|-1)}} ).
\end{align*}
Here $(i)$ follows from  the associativity of $\mu$ and the naturality of $\tau$, $(ii)$ from  Axiom~(GP4), and $(iii)$ from the associativity of $\mu$.  Thus $\Upsilon_{D'}$ is obtained from
$$
\Upsilon_D= \left (\bigotimes_{u \in \up} \phi \mu_{|u|} \right )  P_\sigma S_\low \Delta_\low \Omega^{\otimes g}
$$
by replacing $\phi\mu_{|u_0|}$ with $$\phi\mu_{|u_0|}(f \otimes \id_{A^{\otimes(|u_0|-1)}} ).$$ Consequently,
$$
\Upsilon_{D'} =\phi^{\otimes g} \mu_{\hspace{1pt}\up}  (\id_{A^{\otimes m}} \otimes f \otimes \id_{A^{\otimes n}}) P_\sigma  S_\low \Delta_\low \Omega^{\otimes g}
$$
for some non negative integers $m,n$. Using the naturality of $\tau$ and Lemma~\ref{lem-gp}(c), we obtain (by sliding down $f$) that $\Upsilon_{D'} = \Upsilon_D$. This proves the invariance under the choice of  the base point for the upper circles. The  invariance under the choice of the base point for the lower circles is proved similarly using Axiom (GP5) of a good pair and Lemma~\ref{lem-gp}(d).

Third, let us prove the invariance under the choice of an order for $\up$. Since the transpositions between consecutive numbers generate the symmetric group, we only need to consider the case where we modify the order of $\up$ by exchanging two consecutive upper circles $u_0<u_1$. This transforms $D$ into an ordered oriented based Heegaard diagram $D'$. Then $K_A(D')$ is obtained from
$$
K_A(D)= \left (\bigotimes_{u \in \up} \mu_{|u|} \right )  P_\sigma S_\low \Delta_\low
$$
by replacing $\mu_{|u_0|} \otimes \mu_{|u_1|}$  with
$$
(\mu_{|u_1|} \otimes \mu_{|u_0|})  \tau_{A^{\otimes |u_0|},A^{\otimes |u_1|}}.
$$
Now the naturality of $\tau$ and \eqref{eq-sigma-def} imply that
$$
(\phi\mu_{|u_1|} \otimes \phi\mu_{|u_0|})  \tau_{A^{\otimes |u_0|},A^{\otimes |u_1|}}
=\tau_{I,I} (\phi\mu_{|u_0|} \otimes \phi\mu_{|u_1|})= \sigma_I \, (\phi\mu_{|u_0|} \otimes \phi\mu_{|u_1|})
$$
Thus $\Upsilon_{D'}$ is obtained from
$$
\Upsilon_D= \left (\bigotimes_{u \in \up} \phi \mu_{|u|} \right )  P_\sigma S_\low \Delta_\low \Omega^{\otimes g}
$$
by replacing $\phi\mu_{|u_0|} \otimes  \phi\mu_{|u_1|}$ with
$$
 \sigma_I \, (\phi\mu_{|u_0|} \otimes \phi\mu_{|u_1|}).
$$
Consequently,  $\Upsilon_{D'} = \sigma_I \, \Upsilon_D$.
Since $\sigma_I^{\gamma_{(\phi,\Omega)}}=1$ by the definition of $\gamma_{(\phi,\Omega)}$ and \eqref{eq-sigma-invert-ob}, we conclude that
$$
\norm{\Upsilon_{D'}}_{I^{\otimes g}}^{\gamma_{(\phi,\Omega)}}
=\sigma_I^{\gamma_{(\phi,\Omega)}}\, \norm{\Upsilon_D}_{I^{\otimes g}}^{\gamma_{(\phi,\Omega)}}
=\norm{\Upsilon_D}_{I^{\otimes g}}^{\gamma_{(\phi,\Omega)}}.
$$
This proves the invariance under the choice of an order for $\up$. The  invariance under the choice of an order for $\low$ is proved similarly.

It remains to prove that $\norm{\Upsilon_D}_{I^{\otimes g}}^{\gamma_{(\phi,\Omega)}}$ does not depend on the choice of the Heegaard diagram $D=(\Sigma,\up,\low)$ representing $M$. The Reidemeister-Singer theorem (see \cite[Theorem 8]{Si} or \cite[Theorem 4.1]{Ku1}) asserts that two Heegaard diagrams for~$M$ give rise to homeomorphic oriented 3-manifolds if and only if one can be obtained from the other by a finite sequence of four types of moves (and their inverse), namely:  (I) homeomorphism of the surface, (II) isotopy of the diagram, (III) stabilization, and (IV) sliding a circle past another circle (both upper or both lower). Let us prove the invariance under these moves:\\

{\it Type I.} By using an orientation-preserving homeomorphism from $\Sigma$ to a (closed, connected, and oriented) surface $\Sigma'$, the upper (respectively, lower) circles on $\Sigma$ are carried to the upper (respectively, lower) circles on $\Sigma'$ to obtain a Heegaard diagram $D'=(\Sigma',\up',\low')$. We endow each upper and lower circle of $D'$ with an orientation and a base point and pick  orders for the sets $\up'$ and $\low'$ in such a way that they are preserved by the homeomorphism. Clearly, $K_A(D')=K_A(D)$ and so $\norm{\Upsilon_D'}_{I^{\otimes g}}^{\gamma_{(\phi,\Omega)}}=\norm{\Upsilon_D}_{I^{\otimes g}}^{\gamma_{(\phi,\Omega)}}$.\\

{\it Type II.} We isotop the lower circles of $D$ relative to the upper circles. If
  this isotopy is in general position, it reduces to a sequence of the following two-point moves with $u \in \up$ and $l \in \low$:
$$
 \psfrag{u}[Br][Br]{\scalebox{1}{$u$}}
 \psfrag{v}[Br][Br]{\scalebox{1}{$l$}}
 \rsdraw{.49}{.9}{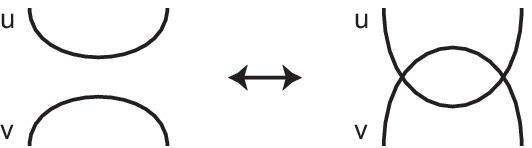} \;\;.
$$
Note that one of the added intersection points is positive and the other one is negative. We can consider that these two circles are oriented so that the invariance is a consequence of
the following equalities:
\begin{align*}
&\;\;\psfrag{A}[Br][Br]{\scalebox{.9}{$F$}}
 \psfrag{B}[Br][Br]{\scalebox{.9}{$G$}}\rsdraw{.45}{.9}{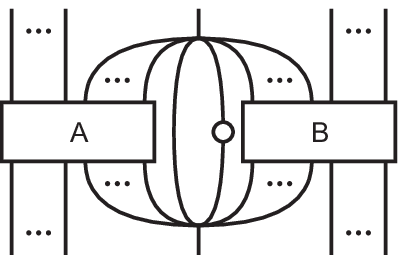} \;\; \overset{(i)}{=} \;\; \rsdraw{.45}{.9}{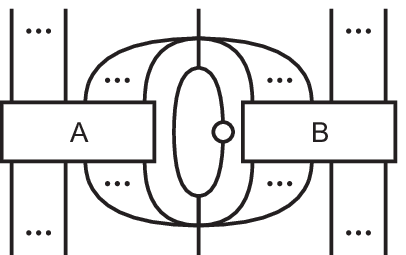} \\[.6em]
\overset{(ii)}{=}&\;\;\psfrag{A}[Br][Br]{\scalebox{.9}{$F$}}
 \psfrag{B}[Br][Br]{\scalebox{.9}{$G$}} \rsdraw{.45}{.9}{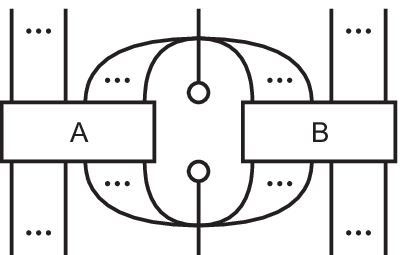}  \; \overset{(iii)}{=} \; \rsdraw{.45}{.9}{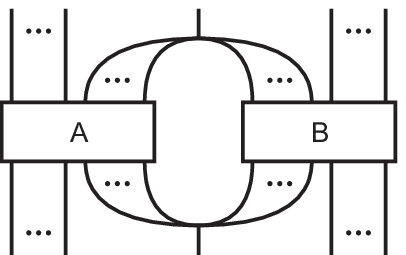}
\end{align*}
where $F$ and $G$ are morphisms between monoidal powers of $A$ (built of the symmetry of $\sss$ and the structural morphisms of the Hopf algebra $A$). Here $(i)$ follows from the (co)associativity of the (co)product, $(ii)$ from the axiom of the antipode, and $(iii)$ from the (co)unitality of the (co)product.\\

{\it Type III.} We remove a disk from $\Sigma$ which is disjoint from all upper and lower circles and replace
  it by a punctured torus with one upper circle and one lower circle. One of them corresponds to the standard
  meridian and the other to the standard longitude of the added torus:
$$
 \psfrag{D}[Bc][Bc]{$D$}
 \rsdraw{.49}{.9}{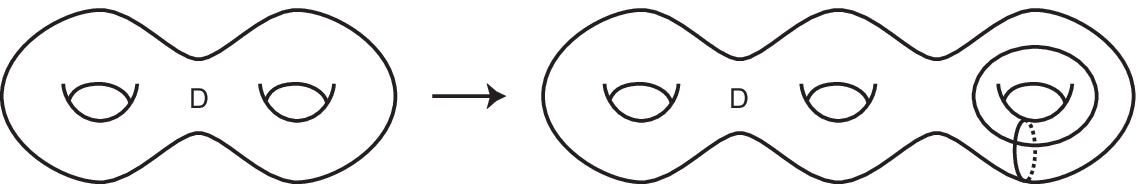} \;\;.
$$
The added upper and lower circles are  endowed with base points arbitrarily  and are oriented in such a way that the added intersection point is positive. Then the invariance follows from the equality $\norm{\phi\Omega}_I=\norm{\id_I}=1$ which is a direct consequence of Axiom (GP1) of a good pair.\\

{\it Type IV.} Let us slide an upper circle $u_1$ past an upper circle $u_2$. We can assume that $u_1$ and $u_2$ are consecutive with $u_1<u_2$. Let $b$ be a band on $\Sigma$ which connects $u_1$ to $u_2$ outside the intersection points (meaning that $b\co I \times I \hookrightarrow \Sigma$ is an embedding such that $b(I \times I)\cap u_i=b(\{i\}  \times I)\subset u_i\setminus \I_{u_i}$ for $i \in \{1,2\}$, where $I=[1,2]$) but does not cross any other circle. The circle $u_1$ is replaced by
\begin{equation*}
  u'_1=u_1 \#_b u_2=u_1 \cup u_2 \cup b(I \times \partial I) \setminus
  b(\partial I \times I).
\end{equation*}
The circle $u_2$ is replaced by a copy $u'_2$ of itself which is slightly isotoped such that it has no point in common with $u'_1$. The new circle $u'_1$ and $u'_2$ inherit the orientations and base points from $u_1$ and~$u_2$, respectively. We choose the base points and orientations of $u_1$ and~$u_2$  so that the move becomes
$$
 \psfrag{C}[Bl][Bl]{\scalebox{1}{$u_1$}}
 \psfrag{D}[Bl][Bl]{\scalebox{1}{$u_2$}}
 \psfrag{U}[Bl][Bl]{\scalebox{1}{$u'_1$}}
 \psfrag{V}[Br][Br]{\scalebox{1}{$u'_2$}}
 \psfrag{b}[Bc][Bc]{$b$}
 \rsdraw{.49}{.9}{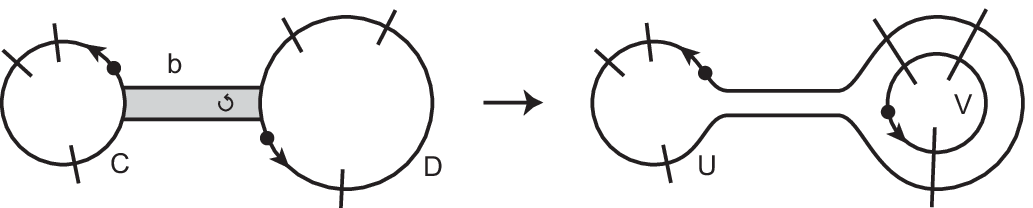}
$$
where the small arcs represents the parts of the lower circles intersecting the involved upper circles and the orientation of the band is induced by that of $\Sigma$.
This transforms~$D$ into an ordered oriented  based  Heegaard diagram $D'$.
Let
$$
\I_{u_2}=\{c_1<\dots<c_n\} \mand \I_{u'_2}=\{c'_1<\dots < c'_n\}
$$
where $n=|u_2|$.
For  $1 \leq i \leq n$, set $\kappa_i=\kappa_{c_i}$ (in the notation of Section~\ref{sect-asso-tensor-HD}).
Note that $\kappa_{c'_i}=\kappa_i$ since the intersection points $c_i$ and $c'_i$ have the same positivity.
For any $\kappa \in \{0,1\}$, define
$$
\delta_\kappa=\tau_{A,A}^\kappa \Delta=\begin{cases}
\Delta & \text{if $\kappa=0$,} \\ \tau_{A,A} \Delta & \text{if $\kappa=1$.}
\end{cases}
$$
It follows from the definitions and the coassociativity of the coproduct that  $K_A(D')$ is obtained from
$$
K_A(D)= \left (\bigotimes_{u \in \up} \mu_{|u|} S_u \right )  P_\sigma  \Delta_\low
$$
by replacing $\mu_{|u_1|}S_{u_1} \otimes \mu_{|u_2|}S_{u_2}$ with
$$
\psfrag{A}[Bc][Bc]{\scalebox{.9}{$S^{\kappa_1}$}}
\psfrag{B}[Bc][Bc]{\scalebox{.9}{$S^{\kappa_n}$}}
\psfrag{C}[Bc][Bc]{\scalebox{.9}{$S_{u_1}$}}
\psfrag{m}[Bc][Bc]{\scalebox{.9}{$\delta_{\kappa_1}$}}
\psfrag{w}[Bc][Bc]{\scalebox{.9}{$\delta_{\kappa_n}$}}
\rsdraw{.45}{.9}{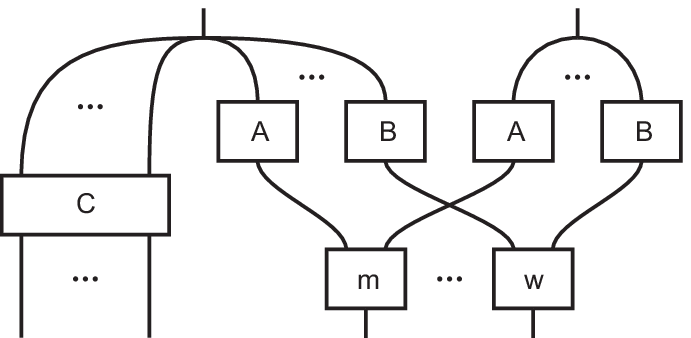}\;\,.
$$
Now
$$
\psfrag{A}[Bc][Bc]{\scalebox{.9}{$S^{\kappa_1}$}}
\psfrag{B}[Bc][Bc]{\scalebox{.9}{$S^{\kappa_n}$}}
\psfrag{D}[Bc][Bc]{\scalebox{.9}{$S_{u_2}$}}
\psfrag{m}[Bc][Bc]{\scalebox{.9}{$\delta_{\kappa_1}$}}
\psfrag{w}[Bc][Bc]{\scalebox{.9}{$\delta_{\kappa_n}$}}
\rsdraw{.45}{.9}{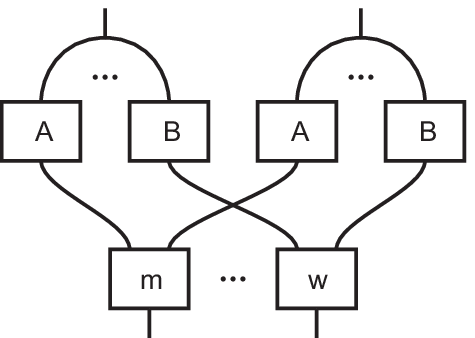} \;\,\overset{(i)}{=}\;\,
\rsdraw{.45}{.9}{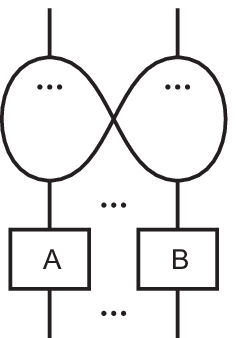} \;\,\overset{(ii)}{=}\;\,
\rsdraw{.45}{.9}{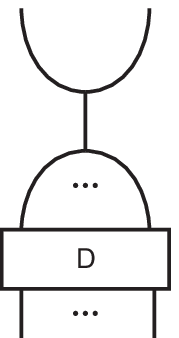}
$$
Here $(i)$ follows from the equality $(S^\kappa \otimes S^\kappa)\delta_\kappa=\Delta S^\kappa$ (which is a consequence of the anti-comultiplicativity of $S$) and $(ii)$ from the multiplicativity of $\Delta$. Thus
$K_A(D')$ is obtained from $K_A(D)$ by replacing $\mu_{|u_1|}S_{u_1} \otimes \mu_{|u_2|}S_{u_2}$ with
$$
\psfrag{C}[Bc][Bc]{\scalebox{.9}{$S_{u_1}$}}
\psfrag{D}[Bc][Bc]{\scalebox{.9}{$S_{u_2}$}}
\rsdraw{.4}{.9}{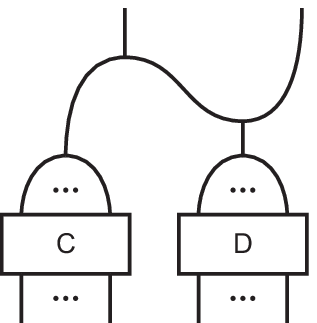}\;\,=\,H(\mu_{|u_1|}S_{u_1} \otimes \mu_{|u_2|}S_{u_2})
$$
where $H=(\mu \otimes \id_A)(\id_A \otimes \Delta)$. Consequently, the equality $(\phi \otimes \phi)H=\phi \otimes \phi$ from Axiom (GP2) of a good pair implies that $\phi^{\otimes g} K_A(D')=\phi^{\otimes g}  K_A(D)$, and so that $\Upsilon_{D'}=\Upsilon_D$.
This proves the invariance under sliding an upper circle past another upper circle. The  invariance under sliding a lower circle past another lower circle is proved similarly using the equality $H(\Omega \otimes \Omega)=\Omega \otimes \Omega$ from Axiom (GP2).
This completes the proof of Theorem~\ref{main-thm}.

\section{Good pairs from (co)integrals}\label{sec4}
In this section, $\sss=(\sss,\otimes,\un)$ is a symmetric monoidal category with symmetry
$\tau=\{\tau_{X,Y} \co X \otimes Y \to Y \otimes\}_{X,Y \in \Ob(\sss)}$. We derive good pairs from integrals and cointegrals that we first define.

\subsection{Integrals and cointegrals} \label{sect-def-integrals}
Consider  a bialgebra $A=(A,\mu,\eta,\Delta,\varepsilon)$ in $\sss$.
A \emph{left integral} of $A$ is a morphism $\Lambda \co I \to A$ in $\sss$ such that
$$
\mu(\id_A \otimes \Lambda)=\varepsilon \otimes \Lambda.
$$
Similarly, a \emph{right integral} of $A$ is a morphism $\Lambda \co \un \to A$ such that
$$
\mu(\Lambda \otimes \id_A)=\Lambda  \otimes \varepsilon.
$$
A \emph{left cointegral} of $A$ is a morphism $\lambda \co A \to I$ such that
$$
(\id_A \otimes \lambda)\Delta=\eta \otimes \lambda.
$$
Similarly, a \emph{right cointegral} of $A$ is a morphism $\lambda \co A \to I$ such that
$$
(\lambda \otimes\id_A)\Delta=\lambda \otimes  \eta.
$$
A left/right integral $\Lambda \co I \to A$ is \emph{universal} if any left/right integral $X \to A$ factorizes as $\Lambda f$ for a unique morphism $f \co X \to I$.  Similarly, a left/right cointegral $\lambda \co A \to I$ is \emph{universal} if any left/right cointegral $A \to X$ factorizes as $g\lambda$ for a unique morphism $g \co I \to X$.

We note two properties of (co)integrals of the bialgebra~$A$ (easily deduced from the definitions). First, multiplying a left/right  integral of~$A$  with an element of $\End_\sss(\un)$ gives left/right integrals of~$A$ and similarly for cointegrals. Second,  suppose   that $A$ is a Hopf algebra  with antipode~$S$.  Then a morphism $\Lambda\co I \to A$ is a left/right integral of $A$ if and only if $S\Lambda$ is a right/left integral of $A$, and in this case   $\Lambda$ is universal if and only if $S\Lambda$ is universal. Similarly,  a morphism $\lambda\co A \to I$  is a left/right cointegral of $A$ if and only if $\lambda S$ is a right/left cointegral of $A$,  and in this case   $\lambda$ is universal if and only if $\lambda S$ is universal.

\subsection{Distinguished morphisms}\label{sect-dist-elts}
Let $A=(A,\mu,\eta,\Delta,\varepsilon,S)$  be a Hopf algebra in~$\sss$. We assume:
\begin{enumerate}
\labeli
\item[(A1):] There exist an invertible object $I$ of $\sss$, a universal left integral $\Lambda \co I \to A$ of~$A$, and a universal right cointegral  $\lambda \co A \to I$ of~$A$ such that $\lambda\Lambda=\id_I$.
\end{enumerate}
Graphically:
\begin{equation}\label{eq-A1}
\psfrag{b}[Bc][Bc]{\scalebox{.8}{$\lambda$}}
\psfrag{d}[Bc][Bc]{\scalebox{.8}{$\Lambda$}}
\rsdraw{.4}{.9}{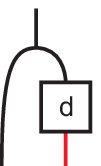}\;=\;\rsdraw{.4}{.9}{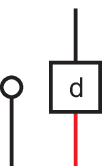}\;\,,
\qquad
\rsdraw{.4}{.9}{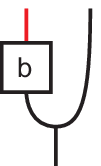}\,=\;\rsdraw{.4}{.9}{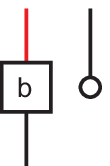}\;\,,
\qquad
\rsdraw{.4}{.9}{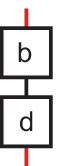}\;=\;\,\rsdraw{.4}{.9}{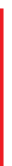}\;\;.
\end{equation}
Here (and below) the color of a black strand is $A$ and that of a red strand is $I$. Note that the assumption (A1) is satisfied for example when $\sss$ is left rigid and all idempotents in~$\sss$ split (see \cite[Theorem 3.3]{BKLV}).

The coassociativity of the coproduct $\Delta$ implies that the morphism
$$
\lambda'=(\id_A \otimes \lambda) \Delta \co A \to A \otimes  I
$$
is a right cointegral of $A$. It follows from the universality of $\lambda$  and the bijection $\Hom_\sss(\un, A) \simeq \Hom_\sss(I,A \otimes I)$ that there is a  unique morphism $g \co \un \to A$ such that  $\lambda'=(g \otimes \id_I)\lambda=g \otimes \lambda$, that is,
\begin{equation}\label{eq-def-g}
\psfrag{u}[Bc][Bc]{\scalebox{.8}{$g$}}
\psfrag{b}[Bc][Bc]{\scalebox{.8}{$\lambda$}}
\rsdraw{.4}{.9}{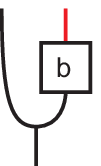}\;=\;\rsdraw{.4}{.9}{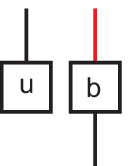}\;.
\end{equation}
It follows from the uniqueness of the factorization through a universal cointegral that the morphism $g$ is grouplike, that is,
$$
\Delta g = g \otimes g \quad \text{and} \quad  \varepsilon g=1.
$$
It is called the \emph{distinguished grouplike element} of $A$.

Likewise, the universality of the integral $\Lambda$ imply the existence of a unique morphism $\alpha \co A \to \un$ such that
\begin{equation}\label{eq-def-alpha}
\psfrag{d}[Bc][Bc]{\scalebox{.8}{$\Lambda$}}
\psfrag{a}[Bc][Bc]{\scalebox{.8}{$\alpha$}}
\rsdraw{.4}{.9}{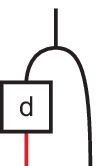}\;=\;\rsdraw{.4}{.9}{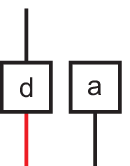}\;.
\end{equation}
Then  $\alpha$ is an algebra morphism, meaning that
$$
\alpha \mu=\alpha \otimes \alpha   \quad \text{and} \quad  \alpha \eta =1.
$$
It is called the \emph{distinguished algebra morphism} of $A$.

The set $\Hom_\sss(\un,A)$ is a monoid with product $(x,y) \mapsto \mu(x \otimes y)$ and with unit~$\eta$. The axioms of an antipode imply that $g$ is invertible in $\Hom_\sss(\un,A)$ with inverse computed by $g^{-1}=Sg$. Similarly, the set $\Hom_\sss(A,\un)$ is a monoid, with product $(f,g) \mapsto (f \otimes g)\Delta$ and with unit $\varepsilon$, in which $\alpha$ is invertible with inverse $\alpha^{-1}=\alpha S$.

The distinguished grouplike element $g$ is said to be \emph{central} if $$\mu(g \otimes \id_A)=\mu(\id_A \otimes g).$$
Note that this implies that $g$ is a central element of the monoid $\Hom_\sss(\un,A)$. Likewise, the distinguished algebra morphism $\alpha$ is said to be \emph{central} if $$(\alpha \otimes \id_A)\Delta=(\id_A \otimes \alpha)\Delta.$$
Note that this implies that $\alpha$ is a central element of the monoid $\Hom_\sss(A,\un)$.

\begin{lemma}\label{lem-ppte-integ-invol}
Assume that $A$ is involutory. Then:
\begin{enumerate}
\labela
\item $\alpha g=1$. More generally, $\alpha^k g^l=1$ for all $k,l \in \Z$.
\item $\lambda\mu(\id_A \otimes g)= \lambda\mu(g \otimes \id_A) = \sigma_I \lambda S$. Graphically:
$$
\psfrag{u}[Bc][Bc]{\scalebox{.8}{$g$}}
\psfrag{b}[Bc][Bc]{\scalebox{.8}{$\lambda$}}
\rsdraw{.4}{.9}{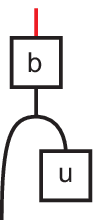}\;=\;\rsdraw{.4}{.9}{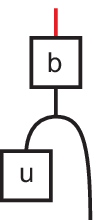}\;=\,\sigma_I\;\rsdraw{.4}{.9}{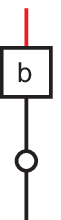}\;.
$$
\item $(\alpha \otimes \id_A)\Delta \Lambda  = (\id_A \otimes \alpha)\Delta \Lambda = \sigma_I S\Lambda$. Graphically:
$$
\psfrag{a}[Bc][Bc]{\scalebox{.8}{$\alpha$}}
\psfrag{d}[Bc][Bc]{\scalebox{.8}{$\Lambda$}}
\rsdraw{.5}{.9}{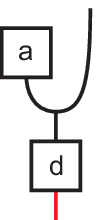}\;=\;\rsdraw{.5}{.9}{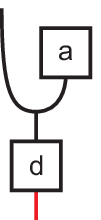}\;=\,\sigma_I\;\rsdraw{.5}{.9}{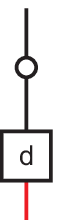}\;.
$$
\item $\lambda \mu \tau_{A,A}=\lambda \mu \bigl((\alpha \otimes \id_A)\Delta \otimes \id_A\bigr)$. Graphically:
$$
\psfrag{a}[Bc][Bc]{\scalebox{.8}{$\alpha$}}
\psfrag{b}[Bc][Bc]{\scalebox{.8}{$\lambda$}}
\rsdraw{.4}{.9}{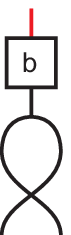}\;=\;\rsdraw{.4}{.9}{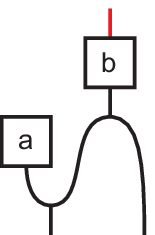}\;.
$$
\item $\tau_{A,A} \Delta \Lambda=\bigl(\id_A \otimes \mu(\id_A \otimes g)\bigr)\Delta \Lambda$. Graphically:
$$
\psfrag{u}[Bc][Bc]{\scalebox{.8}{$g$}}
\psfrag{d}[Bc][Bc]{\scalebox{.8}{$\Lambda$}}
\rsdraw{.4}{.9}{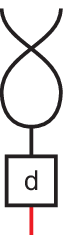}\;=\;\rsdraw{.4}{.9}{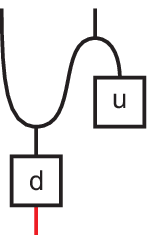}\;.
$$
\item $g$ is central if and only if $\alpha$ is central.
Graphically:
$$
\psfrag{u}[Bc][Bc]{\scalebox{.8}{$g$}}
\rsdraw{.4}{.9}{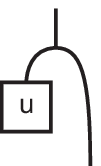}\,=\,\rsdraw{.4}{.9}{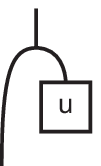}
\quad \Longleftrightarrow \quad
\psfrag{a}[Bc][Bc]{\scalebox{.8}{$\alpha$}}
\rsdraw{.45}{.9}{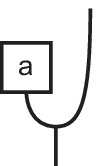}\,=\,\rsdraw{.45}{.9}{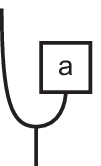}\;\,.
$$
\end{enumerate}
\end{lemma}
We prove Lemma~\ref{lem-ppte-integ-invol} in the Section~\ref{sect-proof-lem-integ-invol}.

\subsection{Good pairs from integrals}
Let $A=(A,\mu,\eta,\Delta,\varepsilon,S)$ be an involutory Hopf algebra in $\sss$. We assume that the assumption (A1) of Section~\ref{sect-dist-elts} is satisfied, i.e., that there are a universal left integral $\Lambda \co I \to A$ and a universal right cointegral  $\lambda \co A \to I$ of~$A$, where $I$ is an invertible object of $\sss$, such that $\lambda\Lambda=\id_I$.
Let us consider the distinguished grouplike element $g \in \Hom_\sss(\un,A)$ and the distinguished algebra morphism $\alpha \in \Hom_\sss(A,\un)$ of $A$ (see Section~\ref{sect-dist-elts}).

In addition to the assumption (A1), we also assume:
\begin{enumerate}
\labeli
\setcounter{enumi}{1}
\item[(A2):] The category $\sss$ is preadditive (meaning that each Hom set has the structure of an abelian group, and composition and monoidal product of morphisms is biadditive).
\item[(A3):] The distinguished morphisms $g$ and $\alpha$ have finite order (in the monoids $\Hom_\sss(\un,A)$ and $\Hom_\sss(A,\un)$).
\item[(A4):] The orders of $g$ and $\alpha$ are invertible in $\End_\sss(\un)$.
\item[(A5):] The distinguished grouplike element $g$ is central or, equivalently, the distinguished algebra morphism $\alpha$ is central (see Lemma~\ref{lem-ppte-integ-invol}(f)).
\end{enumerate}
Denote by $m$ and $n$ the order of $g$ and $\alpha$, respectively.
Set
$$
\phi=\frac{1}{m}\sum_{k=0}^{m-1} \lambda \mu(\id_A \otimes g^k) \co A \to I
\mand
\Omega=\frac{1}{n}\sum_{l=0}^{n-1} (\alpha^l \otimes \id_A)\Delta \Lambda  \co I \to A.
$$
Graphically:
$$
\psfrag{a}[Bc][Bc]{\scalebox{.8}{$\alpha^l$}}
\psfrag{u}[Bc][Bc]{\scalebox{.8}{$g^k$}}
\psfrag{b}[Bc][Bc]{\scalebox{.8}{$\lambda$}}
\psfrag{d}[Bc][Bc]{\scalebox{.8}{$\Lambda$}}
\phi=\frac{1}{m}\sum_{k=0}^{m-1} \; \rsdraw{.4}{.9}{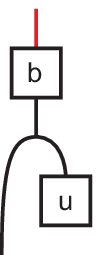}
\mand
\Omega=\frac{1}{n}\sum_{l=0}^{n-1} \; \rsdraw{.45}{.9}{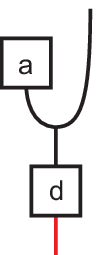} \;.
$$

\begin{theorem}\label{thm-gp}
The pair $(\phi,\Omega)$ is a good pair of $A$ with $\nu_{(\phi,\Omega)}=\sigma_I$.
\end{theorem}
We prove Theorem~\ref{thm-gp} in Section~\ref{sect-proof-thm-gp}.

\subsection{Remarks} 1) Assume that the assumption (A1) is satisfied (recall this is the case when $\sss$ is left rigid and all idempotents in~$\sss$ split). If $\sss$ is $\kk$-linear where $\kk$ is a field (meaning that each Hom set is a $\kk$-vector space and the composition and monoidal product are $\kk$-linear) and if $\Hom_\sss(\un,A)$ and $\Hom_\sss(A,\un)$ are finite dimensional, then the assumptions (A2) and (A3) are satisfied.

2) The assumption (A2) implies that $\End_\sss(\un)$ is a (commutative) ring. Note that if this ring has characteristic zero, then the assumption (A4) is satisfied.

3) If $g=\eta$ or $\alpha=\varepsilon$, then the assumption (A5) is  satisfied. Note that if $g=\eta$ then $\phi= \lambda$, and  if  $\alpha=\varepsilon$ then  $\Omega= \Lambda$.

4) The assumption  (A5) cannot be deduced from the assumptions (A1)-(A4), as  is shown in Example~\ref{sect-c-ex-A5} below.

\subsection{Example}\label{sect-ex-gp-from-integs}
Let $\kk$ be a field, $n$ be a positive integer not divisible by the characteristic of $\kk$, and $\param\in\kk$.
Consider the Hopf algebra $\Hn$ in $\sss=\mathrm{SVect}_\kk$ defined in Example~\ref{sect-ex-super-HA1}. Recall that the $(\Z/2\Z)$-graded vector space $I$ defined by  0 in degree~0 and~$\kk$ in degree 1 is an invertible object of $\mathrm{SVect}_\kk$ such that $\sigma_I=-1$.
Define the morphisms
$$
\lambda \co \Hn \to I \mand \Lambda \co I \to \Hn
$$
by setting, for $0 \leq k \leq n-1$,
$$
\lambda(t^k)=0, \quad \lambda(\ns t^k)=\delta_{k,0} 1_\kk, \mand  \Lambda(1_\kk)=\ns (1+t+ \cdots + t^{n-1}).
$$
Then $\lambda$ is a universal right cointegral of $\Hn$ and $\Lambda$ is a universal two-sided integral of $\Hn$ such that $\lambda \Lambda=\id_I$. The distinguished grouplike element of $\Hn$ is the morphism $g \co \un \to \Hn$ defined by
$$
g(1_\kk)=t.
$$
It has order $n$ in $\Hom_\sss(\un,\Hn)$. The distinguished algebra morphism of $\Hn$ is $\alpha=\varepsilon$ and has order 1 in $\Hom_\sss(\Hn,\un)$. The assumptions (A1)-(A5) are satisfied. Theorem~\ref{thm-gp} yields a good pair $(\phi,\Omega)$ of $\Hn$ given by
$$
\phi=\frac{1}{n}\sum_{k=0}^{n-1} \lambda \mu(\id_{\Hn} \otimes g^k) \mand \Omega=(\alpha^0 \otimes \id_{\Hn})\Delta \Lambda=\Lambda.
$$
For $0 \leq k \leq n-1$, we obtain
$$
\phi(t^k)=0 \mand \phi(\ns t^k)=\frac{1}{n}  1_\kk.
$$
Note that the good pair $(\phi,\Omega)$ is that used in Example~\ref{sect-ex-calcul-inv}.

\subsection{Example}\label{sect-c-ex-A5}
The following example shows that the assumption  (A5) is independent of the assumptions (A1)-(A4).

Let $n=(n_1,n_2)$ be a pair of positive integers with the greatest common divisor $d$. Let $\kk$ be a field and $\omega\in\kk$ be a $d$-th root of unity.  Consider the $\kk$-algebra $A^n_\omega$ generated by $t_1,t_2,\theta_1,\theta_2$
submitted to the relations
\begin{align*}
& t_1^{n_1}=1=t_2^{n_2}, && t_1\theta_1=\theta_1 t_1, && t_1\theta_2=\omega\, \theta_2t_1, && t_1t_2=t_2t_1,\\
& \theta_1^2=0=\theta_2^2, && t_2\theta_2=\theta_2 t_2, && t_2\theta_1=\omega^{-1}\theta_1 t_2, && \theta_1\theta_2=-\omega\, \theta_2\theta_1.
\end{align*}
Note that $A^n_\omega$ is finite-dimensional (over $\kk$) with basis
$$
\bigl\{t_1^kt_2^l,\, \theta_1t_1^kt_2^l,\, \theta_2t_1^kt_2^l,\, \theta_1 \theta_2 t_1^kt_2^l \bigr\}_{0 \leq  k < n_1, \, 0 \leq l < n_2}
$$
and that $t_i$ is invertible with  the  inverse $t_i^{-1}=t_i^{n_i-1}$ for $i\in\{1,2\}$.
By assigning the degrees
$$
|t_1|=|t_2|=0 \mand |\theta_1|=|\theta_2|=1,
$$
we obtain that  $A^n_\omega$ is an algebra in the category $\sss=\mathrm{SVect}_\kk$ of super \kt vector spaces.  Moreover, it becomes an involutory Hopf algebra in $\sss$ with coproduct $\Delta$, counit $\varepsilon$, and antipode $S$ given for all $i\in\{1,2\}$ by
\begin{align*}
& \Delta (t_i)=t_i\otimes t_i, && \varepsilon(t_i)=1, && S(t_i)=t_i^{-1}, \\
& \Delta(\theta_i)=t_i\otimes \theta_i+\theta_i\otimes 1, &&  \varepsilon(\theta_i)=0, && S(\theta_i)=-\theta_i t_i^{-1}.
\end{align*}
Recall that the monoidal unit $\un$ of $\sss$ is defined by $\kk$ in degree~0 and~$0$ in degree 1.
Define the morphisms
$$
\lambda \co A^n_\omega \to \un \mand \Lambda \co \un \to A^n_\omega
$$
by setting
$$
\lambda(t_1^{k}t_2^{l})=\lambda(\theta_1t_1^{k}t_2^{l})=\lambda(\theta_2t_1^{k}t_2^{l})=0, \quad \lambda(\theta_1 \theta_2t_1^{k}t_2^{l})=\delta_{k,0}\,\delta_{l,0}\, 1_\kk,
$$
and
$$
\Lambda(1_\kk)=\sum_{\substack{0\leq k <n_1\\ 0 \leq l <n_2}} \omega^{k-l}\, \theta_1\theta_2 t_1^k t_2^l.
$$
Then $\lambda$ is a universal right cointegral of $A^n_\omega$ and $\Lambda$ is a universal left integral of~$A^n_\omega$ such that $\lambda \Lambda=\id_\un$. The distinguished grouplike element $g \co \un \to A^n_\omega$ of $A^n_\omega$ is defined by
$$
g(1_\kk)=t_1t_2.
$$
Its order in $\Hom_\sss(\un,A^n_\omega)$ is the least common multiple $m=n_1n_2/d$  of $n_1$ and $n_2$. The distinguished algebra morphism $\alpha\colon A^n_\omega\to\un$ of $A^n_\omega$ is defined by
$$
\alpha(t_1)=\omega^{-1}, \quad \alpha(t_2)=\omega, \quad \alpha(\theta_1)=\alpha(\theta_2)=0.
$$
It has order $d$ in $\Hom_\sss(A^n_\omega,\un)$. Thus, the assumptions (A1)-(A4) are satisfied for~$A^n_\omega$.
Moreover, the assumption  (A5) is satisfied (i.e., $g$ is central)  if and only if~$\omega=1_\kk$. Consequently, if $\omega\neq 1_\kk$, then the assumptions (A1)-(A4) are satisfied while the assumption  (A5) is not.

\subsection{Proof of Lemma~\ref*{lem-ppte-integ-invol}}\label{sect-proof-lem-integ-invol}
We first state and prove several claims.

\begin{claim}\label{claim-pr1}
The following equalities hold:
\begin{enumerate}
\labela
\item $(\mu \otimes \id_A)(\id_A \otimes \Delta)=(\id_A \otimes \mu)(\tau_{A,A} \otimes \id_A)(S \otimes\Delta \mu) (\tau_{A,A}\Delta \otimes \id_A)$.\\ Graphically:
$$
\rsdraw{.45}{.9}{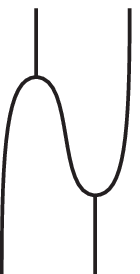}\;=\;\rsdraw{.45}{.9}{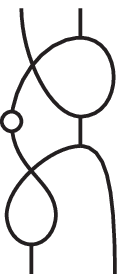}\;\,.
$$
\item $(\id_A \otimes \mu)(\Delta \otimes \id_A)=(\mu \otimes \id_A)(\id_A \otimes \tau_{A,A})(\Delta \mu\otimes S)(\id_A \otimes \tau_{A,A}\Delta)$.\\ Graphically:
$$
\rsdraw{.45}{.9}{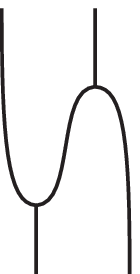}\;=\;\rsdraw{.45}{.9}{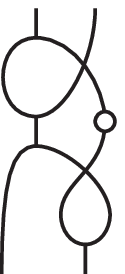}\;\,.
$$
\item $(\id_A \otimes \mu)(\tau_{A,A} \otimes \id_A)(\id_A \otimes \Delta) = (\mu \otimes \id_A) (S \otimes\Delta \mu) (\Delta \otimes \id_A)$.\\ Graphically:
$$
\rsdraw{.45}{.9}{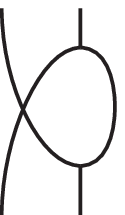}\;=\;\rsdraw{.45}{.9}{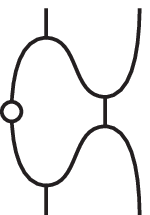}\;\,.
$$
\item $(\mu \otimes \id_A)(\id_A \otimes \tau_{A,A})(\Delta \otimes \id_A)=(\id_A \otimes \mu)(\Delta \mu\otimes S)(\id_A \otimes \Delta)$.\\ Graphically:
$$
\rsdraw{.45}{.9}{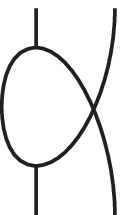}\;=\;\rsdraw{.45}{.9}{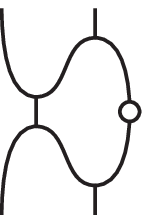}\;\,.
$$
\end{enumerate}
\end{claim}
\begin{proof}
The claim is proved by expanding (in the right hand side of the equalities) $\Delta \mu$ using the multiplicativity of the coproduct, i.e., as
$$
\Delta \mu = (\mu \otimes \mu)( \id_A \otimes \tau_{A,A} \otimes \id_A)(\Delta \otimes \Delta),
$$
and then using the naturality of $\tau$, the (co)associativity and (co)unitality of the (co)product, the axiom of the antipode asserting that
$$
\mu(\id_A \otimes S)\Delta=\eta\, \varepsilon=\mu(S \otimes \id_A)\Delta,
$$
and the  involutivity of the antipode implying that
$$
\mu\tau_{A,A}(\id_A \otimes S)\Delta=\eta\, \varepsilon=\mu\tau_{A,A}(S \otimes \id_A)\Delta.
$$
This latter assertion is proved as follows:
$$
\mu \tau_{A,A}(\id_A \otimes S)\Delta \overset{(i)}{=} \mu \tau_{A,A}(S^2 \otimes S)\Delta \overset{(ii)}{=}
S\mu (S \otimes \id_A)\Delta \overset{(iii)}{=} S\eta\, \varepsilon \overset{(iv)}{=} \eta\, \varepsilon.
$$
Here $(i)$ follows from  involutivity of $S$, $(ii)$ from the multiplicativity of~$S$, $(iii)$ from the axiom of $S$, and $(iv)$ from the unitality of $S$.
\end{proof}

\begin{claim}\label{claim-pr2}
The following equalities hold:
\begin{enumerate}
\labela
\item $(\lambda \mu \otimes \id_A)(\id_A \otimes \Delta \Lambda)=\tau_{A,I}(S \otimes \id_I)$.  Graphically:
$$
\psfrag{b}[Bc][Bc]{\scalebox{.8}{$\lambda$}}
\psfrag{d}[Bc][Bc]{\scalebox{.8}{$\Lambda$}}
\rsdraw{.4}{.9}{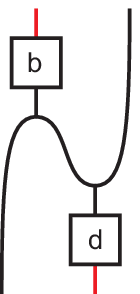}\;=\;\rsdraw{.4}{.9}{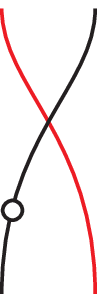}\;\,.
$$
\item $(\id_A \otimes \lambda \mu)(\Delta \Lambda \otimes\id_A)= \tau_{I,A}\bigl(\id_I \otimes \mu(g \otimes \id_A)S(\id_A \otimes \alpha)\Delta \bigr)$. Graphically:
$$
\psfrag{b}[Bc][Bc]{\scalebox{.8}{$\lambda$}}
\psfrag{d}[Bc][Bc]{\scalebox{.8}{$\Lambda$}}
\psfrag{u}[Bc][Bc]{\scalebox{.8}{$g$}}
\psfrag{a}[Bc][Bc]{\scalebox{.8}{$\alpha$}}
\rsdraw{.4}{.9}{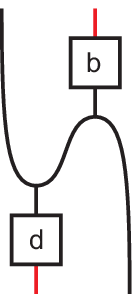}\;=\;\rsdraw{.4}{.9}{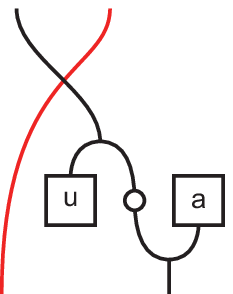}\;\,.
$$
\item $(\id_A \otimes \lambda\mu)(\tau_{A,A} \otimes \id_A)(\id_A \otimes \Delta\Lambda) = (\mu (S\otimes g))\otimes \id_I$. Graphically:
$$
\psfrag{b}[Bc][Bc]{\scalebox{.8}{$\lambda$}}
\psfrag{d}[Bc][Bc]{\scalebox{.8}{$\Lambda$}}
\psfrag{u}[Bc][Bc]{\scalebox{.8}{$g$}}
\rsdraw{.4}{.9}{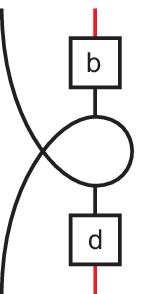}\;=\;\rsdraw{.4}{.9}{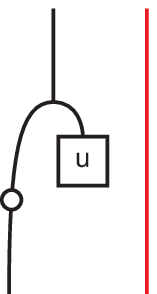}\;\,.
$$
\item $(\lambda\mu \otimes \id_A)(\id_A \otimes \tau_{A,A})(\Delta\Lambda \otimes \id_A)=\id_I \otimes ((\alpha \otimes S)\Delta)$. Graphically:
$$
\psfrag{b}[Bc][Bc]{\scalebox{.8}{$\lambda$}}
\psfrag{d}[Bc][Bc]{\scalebox{.8}{$\Lambda$}}
\psfrag{a}[Bc][Bc]{\scalebox{.8}{$\alpha$}}
\rsdraw{.4}{.9}{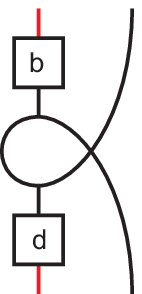}\;=\;\;\rsdraw{.4}{.9}{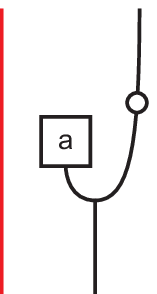}\;\,.
$$
\end{enumerate}
\end{claim}
\begin{proof}
The claim is proved by composing the equalities of Claim~\ref{claim-pr1} with $\lambda \otimes \id_A$ or $\id_A \otimes \lambda$ on the left and with $\Lambda \otimes \id_A$ or $\id_A \otimes \Lambda$ on the right, and then using the naturality of $\tau$ and the properties  \eqref{eq-A1}, \eqref{eq-def-g}, \eqref{eq-def-alpha}.
\end{proof}

Define the morphisms
$$
\psfrag{u}[Bc][Bc]{\scalebox{.8}{$g$}}
\psfrag{b}[Bc][Bc]{\scalebox{.8}{$\lambda$}}
\chi=\rsdraw{.4}{.9}{lg1}=\lambda\mu(\id_A \otimes g)
\mand
\xi=\rsdraw{.4}{.9}{lg2}=\lambda\mu(g \otimes \id_A).
$$

\begin{claim}\label{claim-pr3}
$\chi = \sigma_I  \, \lambda S$.
\end{claim}
\begin{proof}
We have:
\begin{gather*}
\psfrag{b}[Bc][Bc]{\scalebox{.8}{$\lambda$}}
\psfrag{d}[Bc][Bc]{\scalebox{.8}{$\Lambda$}}
\psfrag{o}[Bc][Bc]{\scalebox{.8}{$\lambda\Lambda$}}
\psfrag{u}[Bc][Bc]{\scalebox{.8}{$g$}}
\rsdraw{.4}{.9}{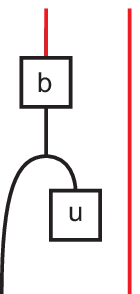}
 \;\;\;\overset{(i)}{=}\; \rsdraw{.4}{.9}{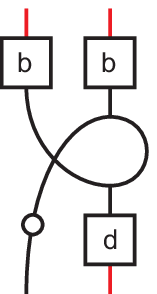}
 \;\;\overset{(ii)}{=}\;\, \rsdraw{.4}{.9}{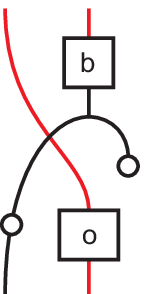}
 \;\;\overset{(iii)}{=}\;\; \rsdraw{.4}{.9}{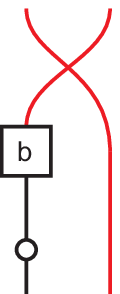}
 \;\,\overset{(iv)}{=}\, \sigma_I\, \rsdraw{.4}{.9}{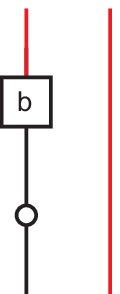}\;\;.
\end{gather*}
Here $(i)$ follows from Claim~\ref{claim-pr2}(c) and the fact that $S^{-1}=S$, $(ii)$ from the naturality of~$\tau$ and the fact that $\lambda$ is a right cointegral, $(iii)$ from the unitality of $\mu$, the fact that $\lambda \Lambda=\id_I$, and the naturality of~$\tau$, and $(iv)$ from \eqref{eq-sigma-def}. Thus $\chi \otimes \id_I=\sigma_I\, \lambda S \otimes \id_I$. Since the map $? \otimes \id_I$ is bijective because the object $I$ is invertible (see Section~\ref{sect-invertible-objects}), we deduce that
$\chi = \sigma_I  \, \lambda S$.
\end{proof}

\begin{claim}\label{claim-pr3b}
$\chi = \sigma_I \alpha g \, \lambda S$.
\end{claim}
\begin{proof}
Consider the morphism
$$
\omega
\psfrag{u}[Bc][Bc]{\scalebox{.8}{$g$}}
\psfrag{a}[Bc][Bc]{\scalebox{.8}{$\alpha$}}
=\;\rsdraw{.4}{.9}{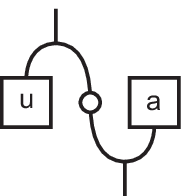}\;=\mu(g \otimes \id_A)S(\id_A \otimes \alpha)\Delta.
$$
The fact that $g$ is grouplike implies that $\omega g=\alpha g \, \eta$. We have:
\begin{gather*}
\psfrag{b}[Bc][Bc]{\scalebox{.8}{$\lambda$}}
\psfrag{d}[Bc][Bc]{\scalebox{.8}{$\Lambda$}}
\psfrag{e}[Bc][Bc]{\scalebox{.8}{$\omega g$}}
\psfrag{u}[Bc][Bc]{\scalebox{.8}{$g$}}
\rsdraw{.4}{.9}{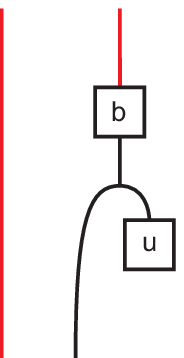}
\;\;\;\overset{(i)}{=}\;\; \rsdraw{.4}{.9}{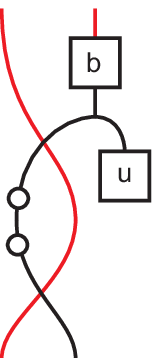}
\;\;\;\overset{(ii)}{=}\;\; \rsdraw{.4}{.9}{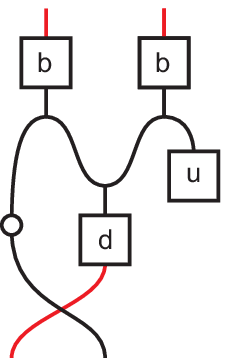}\\[.5em]
\psfrag{b}[Bc][Bc]{\scalebox{.8}{$\lambda$}}
\psfrag{d}[Bc][Bc]{\scalebox{.8}{$\Lambda$}}
\psfrag{e}[Bc][Bc]{\scalebox{.8}{$\omega g$}}
\psfrag{u}[Bc][Bc]{\scalebox{.8}{$g$}}
\overset{(iii)}{=}\;\, \rsdraw{.4}{.9}{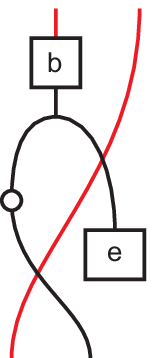}
\;\;\overset{(iv)}{=}\;\alpha g \; \rsdraw{.4}{.9}{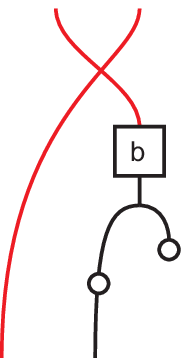}
\;\,\overset{(v)}{=}\; \sigma_I\alpha g\;\;\; \rsdraw{.4}{.9}{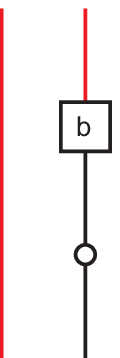}\;\;.
\end{gather*}
Here $(i)$ follows from the properties of the symmetry $\tau$, and the involutivity of $S$,  $(ii)$ from Claim~\ref{claim-pr2}(a), $(iii)$ from Claim~\ref{claim-pr2}(b), $(iv)$ from the naturality of $\tau$ and fact that  $\omega g=\alpha g \, \eta$, and $(v)$ from the unitality of $\mu$ and \eqref{eq-sigma-def}. Thus $\id_I \otimes  \chi = \id_I \otimes (\sigma_I  \alpha g \, \lambda S)$ and so $\chi = \sigma_I  \alpha g \, \lambda S$ (because $I$ is invertible).
\end{proof}

\begin{claim}\label{claim-pr4a}
$\alpha g=1$.
\end{claim}
\begin{proof}
Claims~\ref{claim-pr3} and \ref{claim-pr3b} imply that $\sigma_I \alpha g \, \lambda S=\sigma_I  \, \lambda S$. Now, by Section~\ref{sect-def-integrals}, the morphism $\lambda S$ is a universal  left cointegral (because $\lambda$ is a universal right cointegral). Thus $\sigma_I \alpha g =\sigma_I$. Consequently, using \eqref{eq-sigma-invert-ob}, we deduce that $\alpha g =1$.
\end{proof}

\begin{claim}\label{claim-pr4b}
$\xi=\chi$.
\end{claim}
\begin{proof}
The morphism  $\chi$ is a universal left cointegral since $\chi = \sigma_I  \, \lambda S$ by Claim~\ref{claim-pr3},
$\lambda S$ is a universal  left cointegral (because $\lambda$ is a universal right cointegral), and the scalar $\sigma_I$ is invertible by~\eqref{eq-sigma-invert-ob}. Now, we have:
\begin{gather*}
\psfrag{b}[Bc][Bc]{\scalebox{.8}{$\lambda$}}
\psfrag{n}[Bc][Bc]{\scalebox{.8}{$g^{-1}$}}
\psfrag{u}[Bc][Bc]{\scalebox{.8}{$g$}}
\rsdraw{.4}{.9}{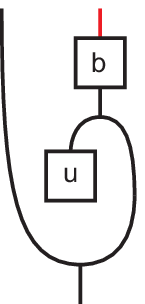}
\;\;\overset{(i)}{=}\; \rsdraw{.4}{.9}{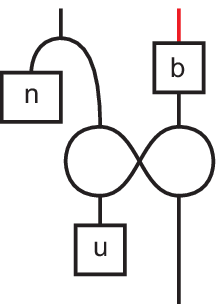}
\;\;\overset{(ii)}{=}\; \rsdraw{.4}{.9}{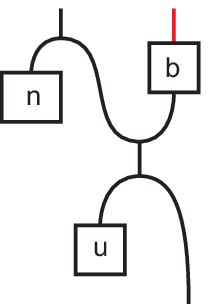}
\;\overset{(iii)}{=}\;\; \rsdraw{.4}{.9}{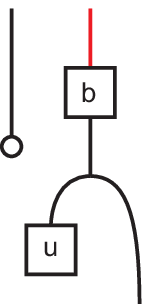}\;.
\end{gather*}
Here $(i)$ follows from the fact that $g$ is grouplike, $(ii)$ from the multiplicativity of~$\Delta$, and $(iii)$ from \eqref{eq-def-g}. Thus $(\id_A \otimes \xi)\Delta=\eta \otimes \xi$, that is, $\xi$ is a left cointegral.
The universality of $\chi$  and the bijection $\End_\sss(I) \simeq \End_\sss(\un)$ imply that there is a  scalar  $x \in \End_\sss(\un)$ such that $\xi=x \, \chi$. Now
\eqref{eq-A1} and the fact that $g$ is grouplike imply that
$$
\xi \Lambda=\lambda \mu (g \otimes \Lambda)=\varepsilon g \, \lambda \Lambda = \id_I.
$$
Also, \eqref{eq-def-alpha} and Claim~\ref{claim-pr4a} imply that
$$
\chi \Lambda=\lambda \mu (\Lambda \otimes g)=\alpha g \, \lambda \Lambda = \id_I.
$$
Consequently, composing the equality $\xi=x \, \chi$ with $\Lambda$ on the right gives $\id_I=x\, \id_I$ and so $x=1$.
Therefore $\xi=\chi$.
\end{proof}

\begin{claim}\label{claim-pr5}
The map $\Psi\co \Hom_\sss(A \otimes A,I)   \to   \Hom_\sss(A \otimes I,I \otimes A)$, defined by
$$
\psfrag{d}[Bc][Bc]{\scalebox{.8}{$\Lambda$}}
\psfrag{n}[Bc][Bc]{\scalebox{.8}{$e$}}
\Psi(e)=\,\rsdraw{.4}{.9}{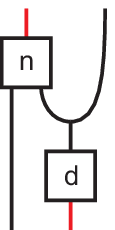}\,= (e\otimes \id_A)(\id_A \otimes \Delta \Lambda)
$$
for any $e \in \Hom_\sss(A \otimes A,I)$, is injective.
\end{claim}
\begin{proof}
Define a map
$\Phi\co \Hom_\sss(A \otimes I,I \otimes A) \to \Hom_\sss(A \otimes I \otimes A,I \otimes I)$ by setting
$$
\psfrag{b}[Bc][Bc]{\scalebox{.8}{$\lambda$}}
\psfrag{n}[Bc][Bc]{\scalebox{.8}{$z$}}
\Phi(z)=\;\rsdraw{.4}{.9}{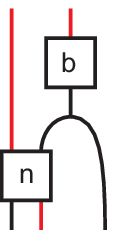}\,=(\id_I \otimes \lambda \mu)   (z \otimes \id_A)
$$
for any $z \in \Hom_\sss(A \otimes I,I \otimes A)$.
Then, for all $e \in \Hom_\sss(A \otimes A,I)$,
$$
\psfrag{b}[Bc][Bc]{\scalebox{.8}{$\lambda$}}
\psfrag{d}[Bc][Bc]{\scalebox{.8}{$\Lambda$}}
\psfrag{n}[Bc][Bc]{\scalebox{.8}{$e$}}
\psfrag{v}[Bc][Bc]{\scalebox{.8}{$\omega$}}
\psfrag{u}[Bc][Bc]{\scalebox{.8}{$g$}}
\psfrag{a}[Bc][Bc]{\scalebox{.8}{$\alpha$}}
\Phi\Psi(e) \overset{(i)}{=}\; \rsdraw{.4}{.9}{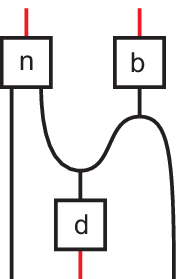}\; \overset{(ii)}{=}\; \rsdraw{.4}{.9}{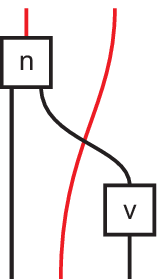} \quad \text{where} \quad
\omega
=\;\rsdraw{.4}{.9}{def-v}\;.
$$
Here $(i)$ follows from the definitions of $\Psi$ and $\Phi$ and $(ii)$ from Claim~\ref{claim-pr2}(b). Now $\omega$
is an isomorphism because $S$ is invertible (as an morphism), $g$ is invertible in the monoid $\Hom_\sss(\un,A)$, and $\alpha$ is invertible in the monoid $\Hom_\sss(A,\un)$. The facts that~$\omega$ and $\tau_{I,A}$ are isomorphims and $I$ is an invertible object imply that $\Phi\Psi$ is injective. Hence, the map $\Psi$ is injective.
\end{proof}

We now prove Lemma~\ref{lem-ppte-integ-invol}.  The first assertion of  Part (a) is Claim~\ref{claim-pr4a},  from which the second assertion  is deduced  by using  the facts that $g$ is grouplike and $\alpha$ is an algebra morphism. Part~(b) is Claims~\ref{claim-pr3} and~\ref{claim-pr4b}.
Part~(c) is just Part (b) applied to the reverse Hopf algebra $A^\reve$ in $\sss^\reve$ (see Section~\ref{sect-dual-reverse}). Indeed, $\lambda$ is a left integral and $\Lambda$ is a right cointegral of $A^\reve$ such that  $\Lambda\lambda=\id_I$ (in~$\sss^\reve$), $\alpha$ is the distinguished grouplike element of $A^\reve$, and $g$ is the distinguished algebra morphism of  $A^\reve$. Let us prove Part~(d). Set
$$
\psfrag{a}[Bc][Bc]{\scalebox{.8}{$\alpha$}}
\psfrag{b}[Bc][Bc]{\scalebox{.8}{$\lambda$}}
d=\;\rsdraw{.4}{.9}{lt1}\;=\lambda \mu \tau_{A,A}
 \mand e=\;\rsdraw{.4}{.9}{lt2}\;= (\alpha\otimes\lambda\mu) \bigl(\Delta\otimes \id_A\bigr).
$$
We need to prove that $d=e$. With the notation of Claim~\ref{claim-pr5}, we have:
\begin{gather*}
\psfrag{b}[Bc][Bc]{\scalebox{.8}{$\lambda$}}
\psfrag{d}[Bc][Bc]{\scalebox{.8}{$\Lambda$}}
\psfrag{a}[Bc][Bc]{\scalebox{.8}{$\alpha$}}
\Psi(d)\,\overset{(i)}{=}\;\;\rsdraw{.45}{.9}{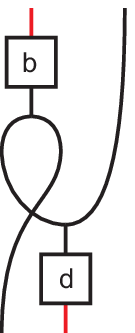}
\;\,\overset{(ii)}{=}\;\;\rsdraw{.45}{.9}{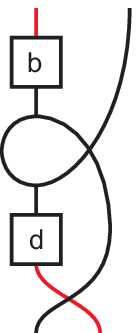}
\,\overset{(iii)}{=}\;\;\rsdraw{.45}{.9}{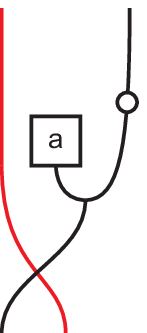}\\[.5em]
\psfrag{b}[Bc][Bc]{\scalebox{.8}{$\lambda$}}
\psfrag{d}[Bc][Bc]{\scalebox{.8}{$\Lambda$}}
\psfrag{a}[Bc][Bc]{\scalebox{.8}{$\alpha$}}
\;\overset{(iv)}{=}\;\;\rsdraw{.45}{.9}{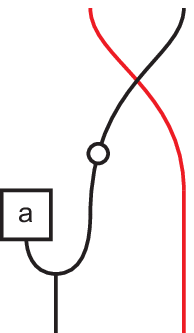}
\;\;\overset{(v)}{=}\;\;\rsdraw{.45}{.9}{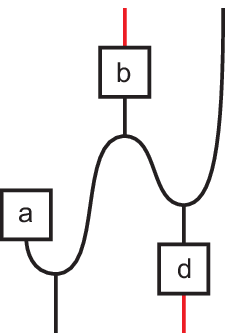}\;\;\overset{(vi)}{=}\,\Psi(e).
\end{gather*}
Here $(i)$ and $(vi)$ follow from the definitions of $\Psi$, $(ii)$ from the properties of the symmetry $\tau$, $(iii)$ from  Claim~\ref{claim-pr2}(d), $(iv)$ from the naturality of $\tau$, and $(v)$ from  Claim~\ref{claim-pr2}(a). Since the map $\Psi$ is injective by Claim~\ref{claim-pr5}, we deduce that $d=e$. This proves Part (d).
Part~(e) is just Part (d) applied to $A^\reve$.

Finally, let us prove Part (f). The invertibility of $g$ in the monoid $\Hom_\sss(\un,A)$ and the associativity of $\mu$ imply that $g$ is central if and only if
$$
\mu_3(g \otimes \id_A \otimes g^{-1})=\id_A.
$$
Similarly, $\alpha$ is central if and only if
$$
(\alpha \otimes \id_A \otimes \alpha^{-1})\Delta_3=\id_A.
$$
Consequently, it suffices  to prove that
$(\alpha \otimes \id_A \otimes \alpha^{-1})\Delta_3=\mu_3(g \otimes \id_A \otimes g^{-1})$ or, equivalently, that
\begin{equation}\label{eq-ga-central-equiv}
(\alpha \otimes \id_A \otimes \alpha^{-1})\Delta_3\mu_3(g^{-1} \otimes \id_A \otimes g)=\id_A.
\end{equation}
Define
$$
\psfrag{u}[Bc][Bc]{\scalebox{.8}{$g$}}
\psfrag{a}[Bc][Bc]{\scalebox{.8}{$\alpha$}}
\omega=\;\rsdraw{.4}{.9}{def-v}\;=\mu(g \otimes \id_A)S(\id_A \otimes \alpha)\Delta.
$$
We have:
\begin{gather*}
\psfrag{b}[Bc][Bc]{\scalebox{.8}{$\lambda$}}
\psfrag{d}[Bc][Bc]{\scalebox{.8}{$\Lambda$}}
\psfrag{v}[Bc][Bc]{\scalebox{.8}{$\omega$}}
\psfrag{u}[Bc][Bc]{\scalebox{.8}{$g$}}
\rsdraw{.4}{.9}{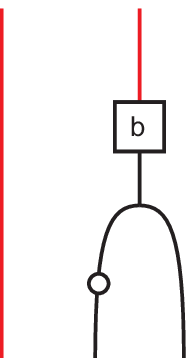}
\;\;\;\overset{(i)}{=}\;\; \rsdraw{.4}{.9}{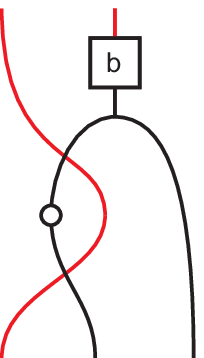}
\;\;\;\overset{(ii)}{=}\;\;\, \rsdraw{.4}{.9}{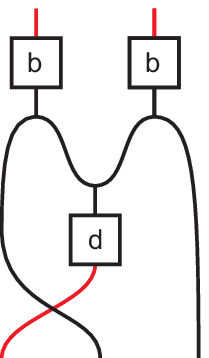}\\[.5em]
\psfrag{b}[Bc][Bc]{\scalebox{.8}{$\lambda$}}
\psfrag{d}[Bc][Bc]{\scalebox{.8}{$\Lambda$}}
\psfrag{v}[Bc][Bc]{\scalebox{.8}{$\omega$}}
\psfrag{u}[Bc][Bc]{\scalebox{.8}{$g$}}
\overset{(iii)}{=}\;\, \rsdraw{.4}{.9}{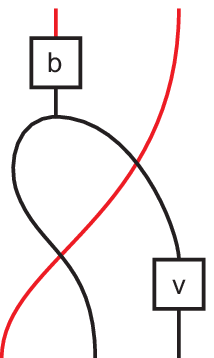}
\;\;\overset{(iv)}{=} \;\, \rsdraw{.4}{.9}{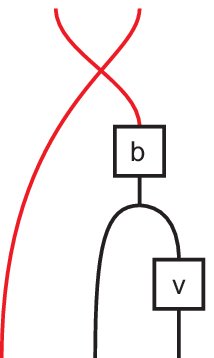}
\;\,\overset{(v)}{=}\;\, \sigma_I\,\;\;\; \rsdraw{.4}{.9}{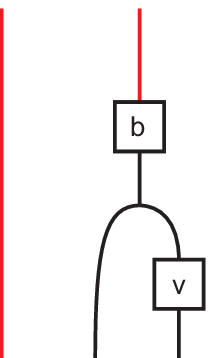}\;\;.
\end{gather*}
Here $(i)$ follows from the properties of the symmetry $\tau$,  $(ii)$ from Claim~\ref{claim-pr2}(a), $(iii)$ from Claim~\ref{claim-pr2}(b), $(iv)$ from the naturality of $\tau$, and $(v)$ from \eqref{eq-sigma-def}. Then, using the fact that the object $I$ is invertible, we obtain that
\begin{equation}\label{eq-ga-central-1}
\psfrag{b}[Bc][Bc]{\scalebox{.8}{$\lambda$}}
\psfrag{v}[Bc][Bc]{\scalebox{.8}{$\omega$}}
\psfrag{u}[Bc][Bc]{\scalebox{.8}{$g$}}
\psfrag{a}[Bc][Bc]{\scalebox{.8}{$\alpha$}}
\rsdraw{.4}{.9}{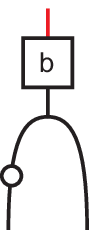}\;=\;\sigma_I \;\; \rsdraw{.4}{.9}{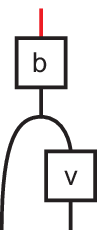} \;\;.
\end{equation}
Now,
\begin{gather*}
\psfrag{b}[Bc][Bc]{\scalebox{.8}{$\lambda$}}
\psfrag{d}[Bc][Bc]{\scalebox{.8}{$\Lambda$}}
\psfrag{v}[Bc][Bc]{\scalebox{.8}{$\omega$}}
\psfrag{a}[Bc][Bc]{\scalebox{.8}{$\alpha$}}
\psfrag{u}[Bc][Bc]{\scalebox{.8}{$g$}}
\psfrag{e}[Bc][Bc]{\scalebox{.8}{$\alpha^{-1}$}}
\psfrag{n}[Bc][Bc]{\scalebox{.8}{$g^{-1}$}}
\rsdraw{.4}{.9}{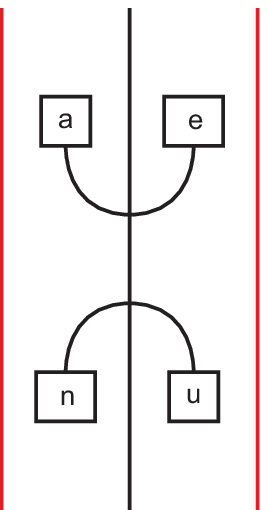}
\,\;\;\;\overset{(i)}{=}\,\;\;\; \rsdraw{.4}{.9}{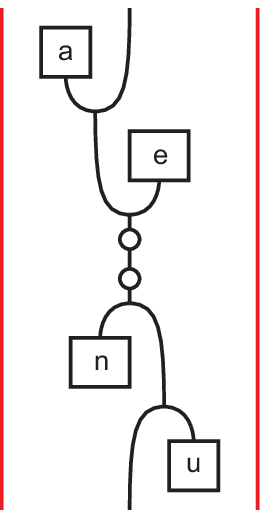}
\,\;\;\;\overset{(ii)}{=}\,\;\;\; \rsdraw{.4}{.9}{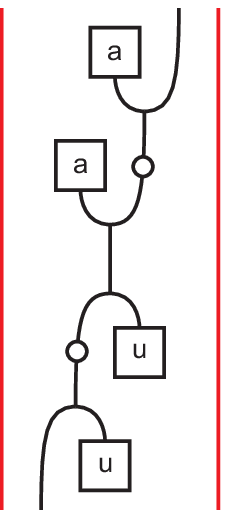}\\[.5em]
\psfrag{b}[Bc][Bc]{\scalebox{.8}{$\lambda$}}
\psfrag{a}[Bc][Bc]{\scalebox{.8}{$\alpha$}}
\psfrag{d}[Bc][Bc]{\scalebox{.8}{$\Lambda$}}
\psfrag{v}[Bc][Bc]{\scalebox{.8}{$\omega$}}
\psfrag{u}[Bc][Bc]{\scalebox{.8}{$g$}}
\overset{(iii)}{=}\;\; \rsdraw{.4}{.9}{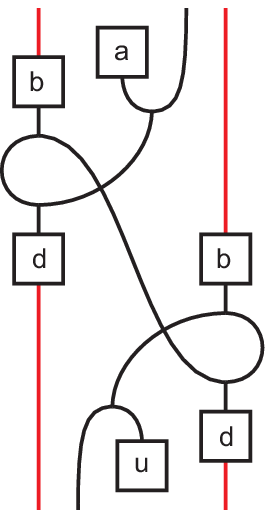}
\;\;\overset{(iv)}{=}\;\;\; \rsdraw{.4}{.9}{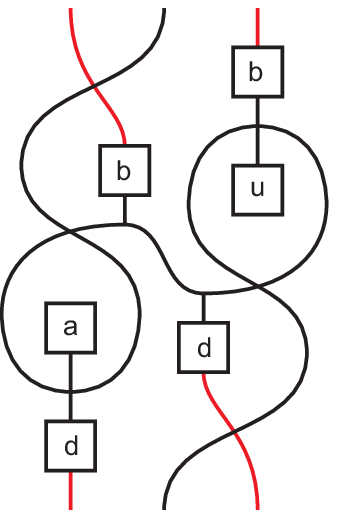}
\;\overset{(v)}{=}\;\; \rsdraw{.4}{.9}{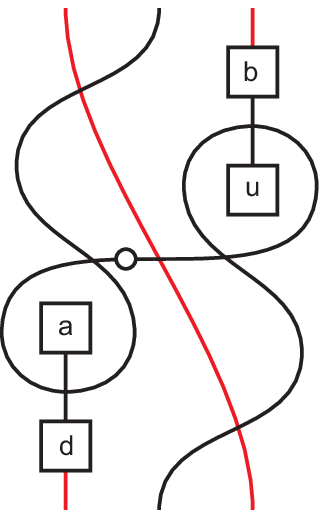}\\[.5em]
\psfrag{b}[Bc][Bc]{\scalebox{.8}{$\lambda$}}
\psfrag{a}[Bc][Bc]{\scalebox{.8}{$\alpha$}}
\psfrag{d}[Bc][Bc]{\scalebox{.8}{$\Lambda$}}
\psfrag{v}[Bc][Bc]{\scalebox{.8}{$\omega$}}
\psfrag{u}[Bc][Bc]{\scalebox{.8}{$g$}}
\overset{(vi)}{=}\;\,\rsdraw{.4}{.9}{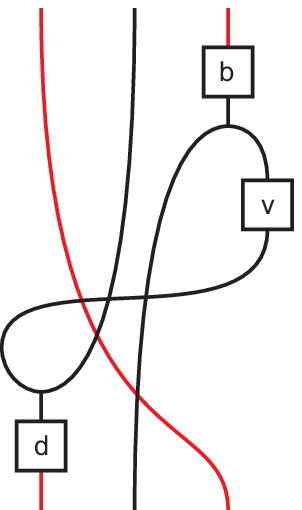}
\;\,\overset{(vii)}{=} \;\, \sigma_I\,\;\; \rsdraw{.4}{.9}{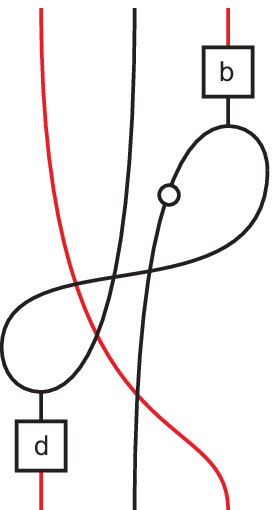}
\,\overset{(viii)}{=}\;\, \sigma_I\,\; \rsdraw{.4}{.9}{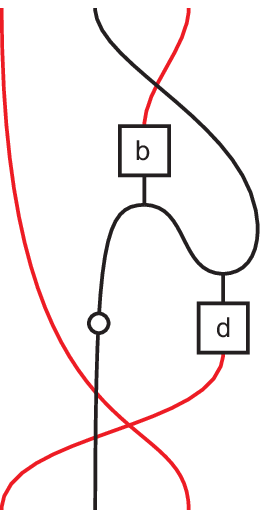}\\[.5em]
\psfrag{b}[Bc][Bc]{\scalebox{.8}{$\lambda$}}
\psfrag{a}[Bc][Bc]{\scalebox{.8}{$\alpha$}}
\psfrag{d}[Bc][Bc]{\scalebox{.8}{$\Lambda$}}
\psfrag{v}[Bc][Bc]{\scalebox{.8}{$\omega$}}
\psfrag{u}[Bc][Bc]{\scalebox{.8}{$g$}}
\overset{(ix)}{=}\;\, \sigma_I^2\,\;\;\; \rsdraw{.4}{.9}{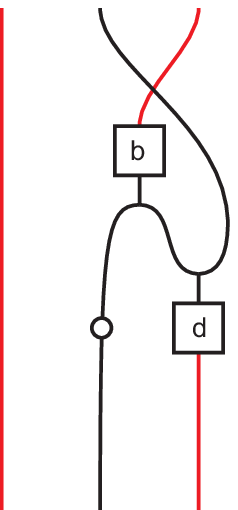}
\;\;\overset{(x)}{=} \,\;\;\; \rsdraw{.4}{.9}{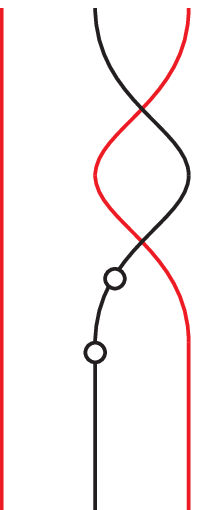}
\;\;\,\overset{(xi)}{=}\,\;\;\; \rsdraw{.4}{.9}{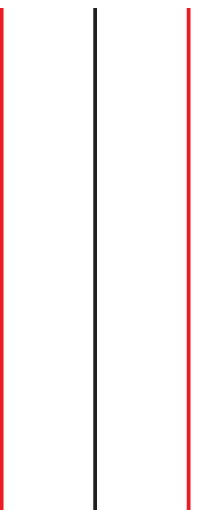}\;\;.
\end{gather*}
Here $(i)$ follows from the (co)associativity of the (co)product and the involutivity of the antipode, $(ii)$ from the computation of the inverse of $g$ an $\alpha$ (see Section~\ref{sect-dist-elts}) and the anti-(co)multiplicativity of the antipode, $(iii)$ from Claim~\ref{claim-pr2}(c) and Claim~\ref{claim-pr2}(d),  $(iv)$ from the (co)associativity of the (co)product and the properties of the symmetry, $(v)$ from  Claim~\ref{claim-pr2}(a), $(vi)$ from  the (co)associativity of the (co)product and the definition of $\omega$,  $(vii)$ from  \eqref{eq-ga-central-1}, $(viii)$ from the properties of the symmetry, $(ix)$ from  \eqref{eq-sigma-def},
$(x)$ from  \eqref{eq-sigma-invert-ob} and  Claim~\ref{claim-pr2}(a), and   $(xi)$ from the involutivity of the antipode and the properties of the symmetry. Then using twice the fact that the object $I$ is invertible, we deduce that  \eqref{eq-ga-central-equiv} is satisfied. This proves Part (f) and concludes the proof of Lemma~\ref{lem-ppte-integ-invol}.

\subsection{Proof of Theorem~\ref*{thm-gp}}\label{sect-proof-thm-gp}
The assumptions (A1)-(A4) ensure that $\phi$ and $\Omega$ are well defined.
We need to verify that the pair $(\phi,\Omega)$ satisfies the axioms (GP1) to (GP5) of a good pair (see Section~\ref{sect-def-good-pair}). Let us prove that $(\phi,\Omega)$ satisfies (GP1). For non-negative integers $k$ and $l$, we have:
\begin{gather*}
\psfrag{a}[Bc][Bc]{\scalebox{.8}{$\alpha^l$}}
\psfrag{u}[Bc][Bc]{\scalebox{.8}{$g^k$}}
\rsdraw{.45}{.9}{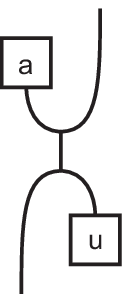}
\overset{(i)}{=}\;\;\,
\rsdraw{.45}{.9}{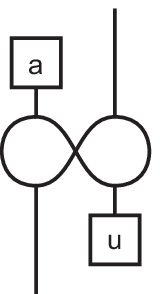}
\;\,\overset{(ii)}{=}\;\alpha^l(g^k)
\rsdraw{.45}{.9}{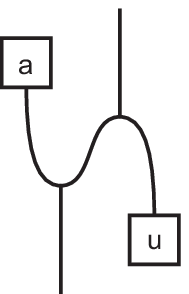}
\overset{(iii)}{=}\;
\rsdraw{.45}{.9}{gpla-3}\;.
\end{gather*}
Here $(i)$ follows from the multiplicativity of $\Delta$, $(ii)$ from the facts that $g$ is grouplike and $\alpha$ is an algebra morphism, and $(iii)$ from Lemma~\ref{lem-ppte-integ-invol}(a). Thus
\begin{gather*}
\psfrag{b}[Bc][Bc]{\scalebox{.8}{$\lambda$}}
\psfrag{d}[Bc][Bc]{\scalebox{.8}{$\Lambda$}}
\psfrag{a}[Bc][Bc]{\scalebox{.8}{$\alpha^l$}}
\psfrag{u}[Bc][Bc]{\scalebox{.8}{$g^k$}}
\rsdraw{.45}{.9}{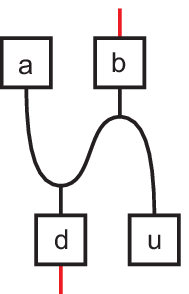}
\;\overset{(i)}{=}\;
\rsdraw{.45}{.9}{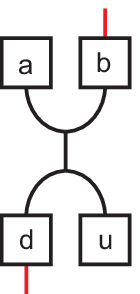}
\;\overset{(ii)}{=}\;\alpha^l(g)\, \alpha(g^k)\;
\rsdraw{.45}{.9}{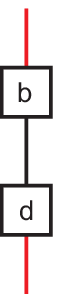}
\;\overset{(iii)}{=}\; \id_I.
\end{gather*}
Here $(i)$ follows from the latter computation, $(ii)$ from \eqref{eq-def-g} and \eqref{eq-def-alpha}, and~$(iii)$ from Lemma~\ref{lem-ppte-integ-invol}(a) and the fact that $\lambda \Lambda=\id_I$. Consequently,
\begin{gather*}
\psfrag{b}[Bc][Bc]{\scalebox{.8}{$\lambda$}}
\psfrag{d}[Bc][Bc]{\scalebox{.8}{$\Lambda$}}
\psfrag{a}[Bc][Bc]{\scalebox{.8}{$\alpha^l$}}
\psfrag{u}[Bc][Bc]{\scalebox{.8}{$g^k$}}
\phi\Omega=\frac{1}{mn}  \sum_{k=0}^{m-1}\sum_{l=0}^{n-1} \;\;\rsdraw{.45}{.9}{gpla-5}\;= \frac{1}{mn}  \sum_{k=0}^{m-1}\sum_{l=0}^{n-1} \id_I=\id_I.
\end{gather*}

Let us prove that $(\phi,\Omega)$ satisfies (GP2). It follows from the definition of $\phi$ and the first equality of Lemma~\ref{lem-ppte-integ-invol}(b) that
\begin{equation}\label{eq-gpla-2}
\phi=\frac{1}{m}\sum_{k=0}^{m-1} \lambda \mu(g^k \otimes \id_A).
\end{equation}
Also, the definition of $\phi$, the associativity of~$\mu$, and the fact that $g$ has order $m$ imply that for any integer $t$,
\begin{equation}\label{eq-gpla-1}
\phi \mu (\id_A \otimes g^t)=\frac{1}{m}\sum_{k=0}^{m-1}  \lambda \mu (\id_A \otimes g^{k+t}) =\phi.
\end{equation}
Now, for any integer $k$, we have:
\begin{gather*}
\psfrag{b}[Bc][Bc]{\scalebox{.8}{$\lambda$}}
\psfrag{n}[Bc][Bc]{\scalebox{.8}{$g^{-k}$}}
\psfrag{u}[Bc][Bc]{\scalebox{.8}{$g^k$}}
\psfrag{a}[Bc][Bc]{\scalebox{.8}{$g^{1-k}$}}
\rsdraw{.45}{.9}{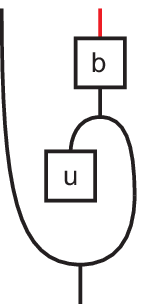}
\;\;\overset{(i)}{=}\;\; \rsdraw{.45}{.9}{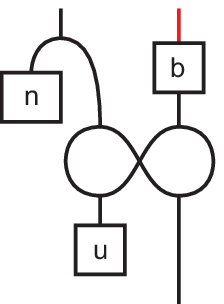}
\;\;\overset{(ii)}{=}\;\, \rsdraw{.45}{.9}{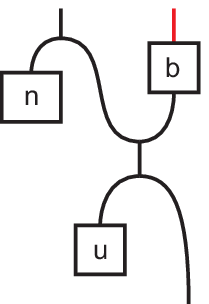}
\;\;\overset{(iii)}{=}\;\, \rsdraw{.45}{.9}{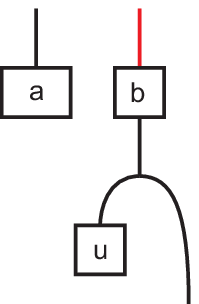}\;.
\end{gather*}
Here $(i)$ follows from the fact that $g$ is grouplike, $(ii)$ from the multiplicativity of the $\Delta$, and $(iii)$ from
\eqref{eq-def-g}. Then
\begin{gather*}
\psfrag{b}[Bc][Bc]{\scalebox{.8}{$\lambda$}}
\psfrag{h}[Bc][Bc]{\scalebox{.8}{$\phi$}}
\psfrag{n}[Bc][Bc]{\scalebox{.8}{$g^{-k}$}}
\psfrag{u}[Bc][Bc]{\scalebox{.8}{$g^k$}}
\psfrag{a}[Bc][Bc]{\scalebox{.8}{$g^{1-k}$}}
\rsdraw{.45}{.9}{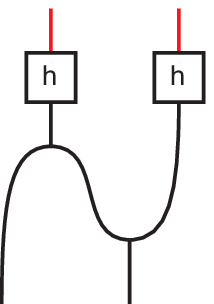}
\;\;\overset{(i)}{=}\;\;\frac{1}{m}\sum_{k=0}^{m-1}\;\; \rsdraw{.45}{.9}{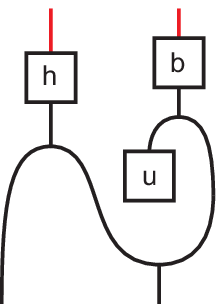}
\;\;\overset{(ii)}{=}\;\,\frac{1}{m}\sum_{k=0}^{m-1}\;\; \rsdraw{.45}{.9}{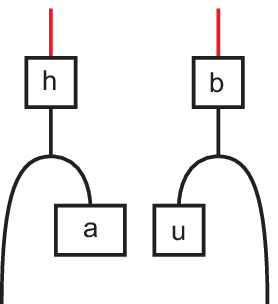}\\[.5em]
\psfrag{b}[Bc][Bc]{\scalebox{.8}{$\lambda$}}
\psfrag{h}[Bc][Bc]{\scalebox{.8}{$\phi$}}
\psfrag{n}[Bc][Bc]{\scalebox{.8}{$g^{-k}$}}
\psfrag{u}[Bc][Bc]{\scalebox{.8}{$g^k$}}
\;\;\overset{(iii)}{=}\;\,\frac{1}{m}\sum_{k=0}^{m-1}\;\; \rsdraw{.45}{.9}{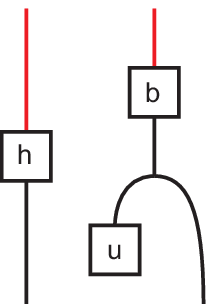}
\;\;\overset{(iv)}{=}\;\;\rsdraw{.45}{.9}{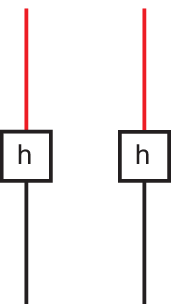}\;\;.
\end{gather*}
Here $(i)$ and $(iv)$ follow from \eqref{eq-gpla-2}, $(ii)$ from the latter computation, and $(iii)$ from~\eqref{eq-gpla-1}.
This proves the first equality of (GP2). The second equality of~(GP2) is just the first equality of (GP2) applied to the reverse Hopf algebra $A^\reve$ (see Section~\ref{sect-dual-reverse})  because $\Omega=\phi_{A^\reve}$.

Let us prove that $(\phi,\Omega)$ satisfies (GP3). We have:
\begin{gather*}
\psfrag{s}[Bc][Bc]{\scalebox{.8}{$g^{1-k}$}}
\psfrag{n}[Bc][Bc]{\scalebox{.8}{$g^{-k}$}}
\psfrag{u}[Bc][Bc]{\scalebox{.8}{$g^k$}}
\psfrag{b}[Bc][Bc]{\scalebox{.8}{$\lambda$}}
\psfrag{d}[Bc][Bc]{\scalebox{.8}{$\Lambda$}}
\phi S
\,\overset{(i)}{=}\;\frac{1}{m}\sum_{k=0}^{m-1}\; \rsdraw{.45}{.9}{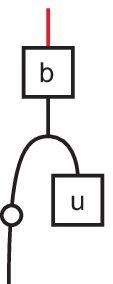}
\;\overset{(ii)}{=}\;\frac{1}{m}\sum_{k=0}^{m-1}\; \rsdraw{.45}{.9}{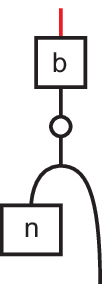}
\,\overset{(iii)}{=}\;\frac{1}{m}\sum_{k=0}^{m-1}\;\sigma_I \; \rsdraw{.45}{.9}{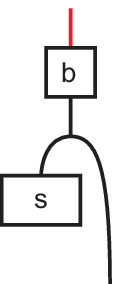}
\;\overset{(iv)}{=}\,\sigma_I \, \phi.
\end{gather*}
Here $(i)$ follows from the definition of $\phi$, $(ii)$ from the anti-multiplicativity of $S$ and the fact that $g$ is grouplike,
$(iii)$ from Lemma~\ref{lem-ppte-integ-invol}(b) and \eqref{eq-sigma-invert-ob}, and $(iv)$ from \eqref{eq-gpla-2} and the fact that $g$ has order~$m$. This proves the first equality of (GP3) because
$$
\nu_{(\phi,\Omega)}\overset{(i)}{=}\norm{\phi S\Omega}_I \overset{(ii)}{=}\sigma_I \norm{\phi \Omega}_I \overset{(iii)}{=}\sigma_I  \norm{\id_I}_I\overset{(iv)}{=}\sigma_I.
$$
Here $(i)$ follows from the definition of $\nu_{(\phi,\Omega)}$, $(ii)$ from the latter computation, $(iii)$ from the fact that $(\phi,\Omega)$ satisfies (GP1), and $(iv)$ from the fact that $\norm{\id_I}_I=1$.
The second equality of~(GP3) is just the first equality of (GP3) applied to  $A^\reve$.

Let us prove that $(\phi,\Omega)$ satisfies (GP4) for the morphism
$$
f=\psfrag{a}[Bc][Bc]{\scalebox{.8}{$\alpha$}} \;\rsdraw{.4}{.9}{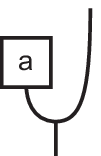}\; =(\alpha \otimes \id_A)\Delta.
$$
First, we have:
\begin{gather*}
\psfrag{o}[Bc][Bc]{\scalebox{.8}{$f$}}
\psfrag{a}[Bc][Bc]{\scalebox{.8}{$\alpha^{l+1}$}}
\psfrag{d}[Bc][Bc]{\scalebox{.8}{$\Lambda$}}
\psfrag{h}[Bc][Bc]{\scalebox{.8}{$\Omega$}}
\rsdraw{.45}{.9}{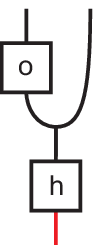}
\;\;\overset{(i)}{=}\;\,\frac{1}{n}\sum_{l=0}^{n-1}\; \rsdraw{.45}{.9}{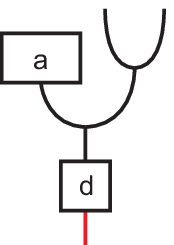}
\,\overset{(ii)}{=}\;\;\rsdraw{.45}{.9}{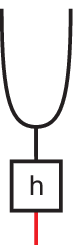} \;.
\end{gather*}
Here $(i)$ follows from the definitions of $\Omega$ and $f$ and the coassociativity of $\Delta$, and~$(ii)$ from the definition of $\Omega$  and the fact that $\alpha$ has order~$n$. Second, we have:
\begin{gather*}
\psfrag{a}[Bc][Bc]{\scalebox{.8}{$\alpha$}}
\psfrag{o}[Bc][Bc]{\scalebox{.8}{$f$}}
\psfrag{u}[Bc][Bc]{\scalebox{.8}{$g^k$}}
\psfrag{b}[Bc][Bc]{\scalebox{.8}{$\lambda$}}
\psfrag{h}[Bc][Bc]{\scalebox{.8}{$\phi$}}
\rsdraw{.45}{.9}{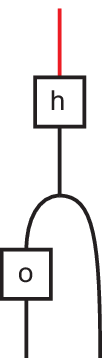}
\;\;\overset{(i)}{=}\;\frac{1}{m}\sum_{k=0}^{m-1}\;\; \rsdraw{.45}{.9}{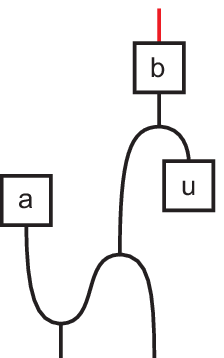}
\;\;\overset{(ii)}{=}\;\frac{1}{m}\sum_{k=0}^{m-1}\;\; \rsdraw{.45}{.9}{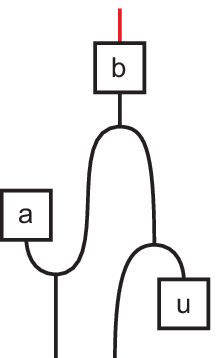}\\[.5em]
\psfrag{a}[Bc][Bc]{\scalebox{.8}{$\alpha$}}
\psfrag{o}[Bc][Bc]{\scalebox{.8}{$f$}}
\psfrag{u}[Bc][Bc]{\scalebox{.8}{$g^k$}}
\psfrag{b}[Bc][Bc]{\scalebox{.8}{$\lambda$}}
\psfrag{h}[Bc][Bc]{\scalebox{.8}{$\phi$}}
\overset{(iii)}{=}\; \frac{1}{m}\sum_{k=0}^{m-1} \rsdraw{.45}{.9}{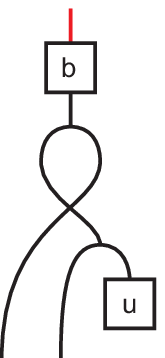}
\!\!\overset{(iv)}{=}\;\frac{1}{m}\sum_{k=0}^{m-1}\;\; \rsdraw{.45}{.9}{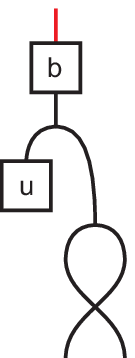}
\;\overset{(v)}{=}\;\; \rsdraw{.45}{.9}{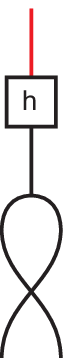} \;.
\end{gather*}
Here $(i)$ follows from the definitions of $f$ and $\phi$, $(ii)$ from the associativity of $\mu$,  $(iii)$ from Lemma~\ref{lem-ppte-integ-invol}(d), $(iv)$ from the naturality of $\tau$ and the fact that $g$ is central by Assumption (A5), and $(v)$ from  \eqref{eq-gpla-2}. This proves (GP4) for the morphism $f$.
Finally, by applying (GP4) to $A^\reve$ we obtain that the pair $(\phi,\Omega)$ satisfies (GP5) for the morphism
$$
h=\psfrag{u}[Bc][Bc]{\scalebox{.8}{$g$}} \;\rsdraw{.4}{.9}{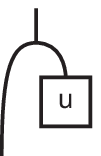}\; = \mu (\id_A \otimes g).
$$
This completes the proof of Theorem~\ref{thm-gp}.

\bibliographystyle{plain}

\end{document}